\documentclass[pdflatex,sn-mathphys-num]{sn-jnl}

\usepackage{graphicx}%
\usepackage{multirow}%
\usepackage{amsmath,amssymb,amsfonts}%
\usepackage{amsthm}%
\usepackage{mathrsfs}%
\usepackage[title]{appendix}%
\usepackage{xcolor}%
\usepackage{textcomp}%
\usepackage{manyfoot}%
\usepackage{booktabs}%
\usepackage{algorithm}%
\usepackage{algorithmicx}%
\usepackage{algpseudocode}%
\usepackage{listings}%
\usepackage[export]{adjustbox}
\usepackage{subcaption}
\usepackage{svg}
\usepackage{ulem}

\theoremstyle{thmstyleone}%
\newtheorem{theorem}{Theorem}
\newtheorem{proposition}[theorem]{Proposition}%

\theoremstyle{thmstyletwo}%
\newtheorem{remark}{Remark}%

\usepackage{caption}

\theoremstyle{thmstylethree}%
\newtheorem{definition}{Definition}%
\newtheorem{lemma}[theorem]{Lemma}

\DeclareMathOperator\supp{supp}
\DeclareMathOperator*{\argmax}{arg\,max}

\newcommand{\Vspace}{L^2(\mu,\Omega;\mathbb{R}^n)}

\begin{document}

\title[Article Title]{Average Optimal Control of Uncertain Control-Affine Systems}


\author[1]{\fnm{M. Soledad} \sur{Aronna}}\email{soledad.aronna@fgv.br}
\equalcont{These authors contributed equally to this work.}

\author*[1]{\fnm{Gabriel} \sur{de Lima Monteiro}}\email{g.delimamonteiro@gmail.com}
\equalcont{These authors contributed equally to this work.}

\author[1]{\fnm{Oscar} \sur{Sierra Fonseca}}\email{oscar.fonseca@fgv.br}
\equalcont{These authors contributed equally to this work.}

\affil[1]{\orgdiv{Escola de Matem\'atica Aplicada}, \orgname{Fundaç\~ao Getulio Vargas}, \orgaddress{\street{Praia de Botafogo 190}, \city{Rio de Janeiro}, \postcode{22250-900}, \state{RJ}, \country{Brazil}}}




\abstract{This work studies optimal control problems of systems with uncertain, probabilistically distributed parameters to optimize average performance. Known as {\it Riemann-Stieltjes, average,} or {\it ensemble optimal control,} this kind of problem is crucial when parameter uncertainty matters. We derive necessary optimality conditions and characterize feedback controls for control-affine systems. Two scenarios are examined: known initial conditions (finite-dimensional case) and uncertain initial conditions (infinite-dimensional framework). The Pontryagin Maximum Principle is extended using a Hilbert space formulation.}




\maketitle

\section{Introduction}
This article studies optimal control problems in which a probability distribution describes the parameters. Our objective is to optimize the control strategy with respect to the average performance.
This class of problems, where the system possesses internal stochasticity, is nowadays known as {\it Riemann-Stieltjes optimal control problems}, {\it average optimal control} or {\it optimal ensemble control problem} (see {\it e.g., }\cite{Zuzua-Loheac,Ross2015,Phelps2016}). Our study focuses on the necessary conditions and the characterization of controls in feedback form for control-affine problems, which appear in several applications \cite{almeida2022optimal,Martinon2009,seyde2021bang} and for which adding uncertainty in the parameters provides significant practical insight.
We consider an optimal control problem where the dynamics involves unknown parameters  within a given metric space. Two cases are studied: first, when the initial data is known, leading to a formulation in a finite-dimensional space, and second, when the initial data depends on uncertain parameters, meaning that the initial state is not fully known, resulting in a formulation within an infinite-dimensional space. More precisely, we are dealing with the following problem
\begin{equation*}
(P) \parbox{.96\textwidth}{
\begin{align}
&\text{minimize} \quad  J[u(\cdot)]:=\int_\Omega g(x(T,\omega),\omega) \, d\mu(\omega), \nonumber \\
&\text{s.t.} \nonumber \\
& \begin{cases}\label{dynamics}
\dot{x}(t,\omega) = f_0(x(t,\omega),\omega) + \displaystyle \sum_{i=1}^m f_i(x(t,\omega),\omega)u_i(t),\\ \hspace{4.2cm}\text{a.e. } t\in [t_0,T], \ \omega\in \Omega, \\
x(t_0,\omega)=\varphi(\omega), & \hspace{-4.1cm} \omega\in \Omega, \\
u(t) \in U(t), & \hspace{-4.1cm}\text{a.e. } t\in [t_0,T],
\end{cases}
\end{align}
}\nonumber
\end{equation*}
where, $(\Omega,\rho_{\Omega})$ is a metric space, $\mu$ a probability measure and   $f_i\colon \mathbb{R}^n\times\Omega\rightarrow\mathbb{R}^n$ for $i=0,\dots ,m\leq n,$ $g 
\colon \mathbb R^n \times \Omega \rightarrow
\mathbb R$ and  $u:[t_0,T]\rightarrow \mathbb{R}^m$ is the control function, which belongs to the {\it admissible controls set}
\begin{equation}\label{Uad}
\mathcal{U}_{\rm ad}[t_0,T] := \left\{ u \in  L^\infty(t_0,T;\mathbb{R}^m): u(t)\in U(t) \quad\text{a.e. } t\in [t_0,T]\right\},
\end{equation}
with  $U :[0,T] \rightsquigarrow \mathbb R^m$ being a multifunction taking non-empty and closed values contained in a compact subset $\textbf{  U}\subset \mathbb{R}^m.$

We investigate necessary optimality condition for problem $(P)$ under two scenarios: {\it Known Initial Condition}, where the initial state $\varphi \equiv x_0 \in \mathbb{R}^n$ is precisely defined, 
and {\it Uncertain Initial Condition}, where $\varphi: \Omega \rightarrow \mathbb{R}^n$ is a measurable function, {\it i.e.,} the initial state is uncertain.

Beginning with \cite{Ross2015}, the Riemann-Stieltjes framework found several applications, as in the aerospace \cite{GonzlezArribas2018,Shaffer2016}, military planning \cite{Walton2018} and affine dynamics for reinforcement learning and system identification \cite{PesareMAY2021,pesare2021convergence}. Numerical results were developed in \cite{Lambrianides2020,Phelps2016}, while \cite{Bettiol2019} extended the theory for systems with uncertainty in the dynamics, cost, and for fixed initial conditions, including versions of the Pontryagin Maximum Principle for smooth and non smooth settings.

In \cite{PesareMAY2021,pesare2021convergence}, the authors explored the usefulness of the Riemann-Stieltjes framework for system identification. Specifically, in \cite{pesare2021convergence}, they applied the framework to the Linear Quadratic Regulator with affine control, addressing cases where the dynamics exhibit uncertainty in the linear drift term. In \cite{PesareMAY2021}, they extended the approach to purely nonlinear dynamics. In this work, we study the theory of the affine setting, that is, we consider uncertainty in a purely nonlinear drift term and introduce uncertainty in the controlled part of the dynamics, similar to the initial condition. 

Necessary conditions for parameterized optimal control problems with known initial conditions have been studied in several papers. For example, in \cite{minimaxVinter}, the necessary conditions for minimax problems with a compact metric space parameter set $\Omega$ are derived. The approach involves approximating $\Omega$ with a finite subset $\Omega_N$, analyzing the properties of approximate minimizers for problems defined on $\Omega_N$, and then taking the limit as $N$ increases. In \cite{Bettiol2019}, the authors deduce first order necessary optimality conditions (in the form of a Pontryagin Maximum Principle) for problem $(P)$ when the uncertainty dynamics is non-linear. 

The problem $(P)$ with an uncertain initial condition was previously investigated in \cite{Zuzua-Loheac,Phelps2016,Scagliotti2023}.  
In \cite{Zuzua-Loheac} it is addressed as a penalized optimization problem to tackle {\it approximate simultaneous controllability} (also referred to as {\it ensemble controllability}) in a weighted $L^2$-space. The authors consider linear uncertain dynamics, {\it i.e.,} $$f_0(x,\omega) = A(\omega)x \quad\text{and} \quad\sum f_i(x,\omega)u_i(t)=B(\omega)u(t),$$  where for every $\omega \in \Omega$, $A(\omega) \in \mathcal{L}(\mathbb{R}^n)$ and $B(\omega) \in \mathcal{L}(\mathbb{R}^m,\mathbb{R}^n)$, and a penalized cost given by 
{\small \begin{equation*}
g(x(T,\omega),\omega)=\frac{1}{2\,
}\int_0^T |u(t)|_{\mathbb{R}^m}^2 \, dt + k  |x(T,\omega) - \psi(\omega)|_{\mathbb{R}^n}^2  \quad k \geq 0,
\end{equation*}}
with the final constraint
$
\int_\Omega x(T,\omega) \, d\mu(\omega) = \int_\Omega \psi(\omega) \, d\mu(\omega),
$
where $x(\cdot, \omega)$ is the solution of \eqref{dynamics} with control $u$ and initial condition $\varphi(\omega)$. In \cite{Phelps2016}, the authors termed this type of problem an {\it uncertain optimal control problem} (UOCP). They considered uncertain initial conditions of the form $\varphi(\omega) = \xi + l(\omega)$, where the goal is to find a pair $(u, \xi)$ that minimizes the cost functional $J$ in problem $(P)$. They derived necessary conditions for both the original and approximate problems, demonstrating that the approximations using sample averages are consistent as defined by Polak \cite{Polak}. This guarantees that the global minimizers of the approximate problems converge to those of the original problem, and through numerical examples, the authors show that their framework facilitates solving both optimal search and optimal ensemble control problems. In \cite{Scagliotti2023} it is considered $\Omega$ as a compact subset of a Euclidean space, and $\mu$ is a Borel probability measure. The approach is similar to that used in \cite{minimaxVinter}; the author addresses the optimal control problem of affine control system ensembles, establishing convergence for approximating infinite ensembles with finite sub-ensembles. The author derive a Maximum Principle for ensemble control and propose numerical algorithms. These algorithms are validated through simulations on 2D linear system ensembles, demonstrating their effectiveness.

This work lies within the established literature on optimal control under uncertainty, extending prior contributions in key directions. We investigate the existence of optimal controls for problem $(P)$ and derive the PMP for both known and uncertain initial conditions. For the known initial condition case (Section \ref{KIC}), our analysis focuses on the general PMP framework of \cite{Bettiol2019} to the class of control-affine systems. This structural assumption enables a concrete characterization of optimal controls, leading to explicit formulas that facilitate a detailed analysis of singular arcs in Section \ref{char}. Our main contribution, however, addresses the uncertain initial condition case (Section \ref{UIC}), where we derive a PMP within an infinite-dimensional Hilbert space setting. Our approach treats the initial condition as an element of this space, meaning the state $x(t, \cdot)$ is also Hilbert-space valued for each $t$. Crucially, this derivation does not rely on a linear semigroup structure, a common restriction in works like \cite{Frankowska1992} and \cite{liYong}, thereby significantly expanding the theory's applicability to purely nonlinear dynamics. The resulting PMP enables us to characterize singular arcs for both scalar and vector-valued controls with commuting vector fields. The special case with known initial condition and scalar-valued control was concisely presented in \cite{AronnaMonteiroSierra2025CDC}. \\ \\ The present work is organized as follows. Section \ref{KIC} studies the problem for known initial conditions, addressing the continuity of trajectories, existence of optimal controls, and revising the necessary conditions from \cite{Bettiol2019} for the control-affine case. Section \ref{UIC} is dedicated to the case of unknown initial conditions, where we discuss the existence of optimal controls and derive the Pontryagin maximum principle for the system. Section \ref{char} deals with the characterization of singular arcs for both scalar control and vector-valued controls with commutative vector fields. Finally, Section \ref{NUM_EXMP} is devoted to numerical implementation, where we first comment on convergence and then present numerical experiments on two examples that validate our theoretical results and prove their practical relevance, leveraging the sample average approximation framework of \cite{Phelps2016}.

\section{Preliminaries}\label{Preliminaries}
We let $\mathcal B^n$ denote the Borel subsets of $\mathbb R^n.$ We write $\mathcal B$- and $\mathcal L$- to refer to Borel- and Lebesgue-measurable functions, respectively. The support
of a measure $\mu$ defined on $\Omega$ is written $\supp(\mu).$ If $X$ is a Banach space, 
we will denote $\mathbb{B}$ as the open unit ball centered at the origin of $X$. By $L^p(0,T;X)$ we mean the Lebesgue space with domain equal to the interval $[0,T]\subset \mathbb{R}$ and with values in $X.$ The notation $W^{q,s}([0,T];X)$ refers to the Sobolev spaces (see {\it e.g.} Adams \cite{adams1975}). For $q=r=1,$ we let the corresponding norm be $\|x\|_{W^{1,1}}:= \|x\|_{1} + \|\dot x\|_{1}$. According to \cite{evans2010}, for any $ x \in W^{1, p}([0, T]; X) $, there exists a continuous function $ y \in C([0, T]; X) $ that represents $ x $, meaning that $ x(t) = y(t) $ for almost every $ t \in [s, T] $.

The graph of a function $f$ is written as ${\rm Gr } f$.
A function $\theta:[0,\infty)\rightarrow[0,\infty)$ that is increasing and such that $\lim_{s\rightarrow 0^+} \theta(s)=0$ is called {\it modulus of continuity}.
Given two smooth vector fields $ f, g: \mathbb{R}^n \rightarrow \mathbb{R}^n $, the {\it Lie bracket} $[f, g]$ is a new vector field defined as:
$$
[f, g] := g'f - f'g.
$$
 We will adopt the following set of assumptions, which we refer to hereafter with \textbf{ (H0)}:
\begin{itemize}
\item[(i)] $\mu$ is a probability measure defined over a complete separable metric space   $(\Omega,\rho_\Omega)$ and, for a.e $t\in [t_0,T],$ $U(t)$ is a non-empty, compact and convex subset of $\mathbb R^m$ and its graph  $Gr U(\cdot)$  is a $ \mathcal{L}\times\mathcal{B}^m$-measurable set.
\item [(ii)] For every  $i=0,1,...,m,$ there exist constants $c_i,k_i>0$  such that, for every  $x,x'\in \mathbb{R}^n,$ $\omega\in \Omega,$  $f_i$  satisfies
\begin{equation}
\label{fbounded}
|f_i(x,\omega)| \leq c_i  (1+|x|)
\end{equation}
and
\begin{equation}
\label{fLipschitz}
|f_i(x,\omega)-f_i(x',\omega)|\leq k_i|x-x'|, \quad \text{for all} \quad i=0,1,...,m.
\end{equation}
\item [(iii)] The function $g:\mathbb{R}^n\times\Omega\rightarrow \mathbb{R}$ is $\mathcal{B}^n\times\mathcal{B}_\Omega$-measurable and there exist positive constants $k_g\geq 1$ and $M$ such that, for all  $x,x'\in \mathbb{R}^n,$ $\omega\in\Omega,$
\begin{equation*}
\begin{aligned}
 |g(x,\omega)|&\leq M, 
\\
    |g(x,\omega)-g(x',\omega)| &\leq k_g|x-x'|.
    \end{aligned}
    \end{equation*}
\end{itemize}

An \textit{admissible process} $(u, \{x(\cdot,\omega) : \omega\in\Omega\})$ consists of a control function $u$ belonging to $\mathcal{U}_{ad}[t_0,T]$, coupled with its associated family of arcs $\{x(\cdot,\omega)\in W^{1,1}([0,T], \mathbb R^n):\omega\in\Omega\}$ such that, for each $\omega\in\Omega$,  $x(\cdot,\omega)$ is solution of the following evolution integral equation
\begin{equation}\label{trajectory}
\begin{aligned}
x(t,\omega)
=\varphi(\omega)+\int_{t_0}^{t} f_0(x(\sigma,\omega),\omega) + \sum_{i=1}^mf_i(x(\sigma,\omega),\omega)\,u_i(\sigma)\,d\sigma.
\end{aligned}
\end{equation} If there exists a  process $(\bar{u}, \{\bar{x}(\cdot,\omega) : \omega\in\Omega\})$ solving problem $(P),$ then it is called   {\it optimal pair}, and we will refer to $\bar{x}$ and $\bar{u}$ as {\it optimal trajectory} and {\it control}, respectively. 

\section{The case of known initial condition}\label{KIC}

We first consider the problem where the initial conditions are known, {\it i.e.,}  problem $(P)$  with $\varphi(\omega) = x_0$ for all $\omega\in\Omega$. 
It is sometimes useful to define a generic expression for the system dynamics. Hence, for the appropriate cases, we consider the function $f$ given by
\begin{equation}\label{dyna.finito}
f(x,u,\omega):=f_0(x,\omega) + \sum_{i=1}^m f_i(x,\omega)u_i.
\end{equation}


\subsection{Continuity of trajectories}

This subsection aims at studying the regularity of the trajectories with respect to the parameters and controls. In the next result, for each control $u \in \mathcal{U}_{ad}[t_0,T], $ the associated trajectory will be denoted by $x_u(\cdot,\omega)$ for all $\omega\in \Omega.$
\begin{lemma}\label{cont.traj}
   Suppose that \textbf{  (H0)} holds true, then the following properties are verified.
    \begin{itemize}
        \item[(i)] For each $u \in \mathcal{U}_{ad}[t_0,T], $ the trajectory $x_u$ is uniformly bounded: $$|x_u(t,\omega)|\leq C_x \quad \text{for all } t\in [t_0,T], \ \omega \in \Omega,$$where $C_x $ is a positive constant which depends on $u$ and on $x_0.$
        \item[(ii)] If there exists a modulus of continuity $\theta_{f}$ such that, for all  \ $\omega_1,\omega_2\in\Omega,$
    {\small\begin{equation}\label{f.regular}
        \int_0^T \sup_{x\in\mathbb{R}^n, u\in U(t)} |f(x,u,\omega_1) - f(x,u,\omega_2)| dt \leq \theta_{f}(\rho_\Omega(\omega_1,\omega_2)),
    \end{equation}}
     then,  for each $u \in \mathcal{U}_{ad}[t_0,T],$ the map $\omega\mapsto x_u(\cdot,\omega)$ is continuous from $\Omega$ to $W^{1,1}([t_0,T];\mathbb R^n)$;
        \item [(iii)] For each $\omega\in \Omega,$ the map $u\mapsto x_u(\cdot,\omega)$ is
        \begin{itemize}
            \item[$\bullet$] continuous   from ${\mathcal{U}}_{\rm ad}[t_0,T]$ to $W^{1,1}([t_0,T];\mathbb R^n),$
            \item[$\bullet$] continuous  from $L^2(t_0,T;\mathbb{R}^m)$ to $W^{1,1}([t_0,T];\mathbb R^n).$ 
        \end{itemize}
    \end{itemize}
\end{lemma}
\begin{proof} See Appendix \ref{AppA}.
\end{proof}

\subsection{Existence of optimal control}
To establish results on the existence of optimal controls and necessary optimality conditions for the problem $(P)$ with a given initial condition, we first note that due to the compactness and convexity of $U(t)$ for almost every $t \in [t_0, T]$, and the affine structure of the dynamics, for all $x\in \mathbb{R}^n,$ $\omega \in \Omega,$ the set 
\begin{equation}\label{Comp.Conv}
f(x,U(t),\omega)
\ \text{is compact and convex for {\it a.e.} }\ t \in [t_0,T]. 
\end{equation}
In the proof of the following existence result, we will make essential use of \eqref{Comp.Conv} and the sequential compactness of trajectories established in \cite[Theorem 23.2]{Clarke}.

\begin{theorem}[Existence of optimal control]\label{Existence}
Let us suppose that set of conditions \textbf{  (H0)} 
is fulfilled. Then, if there is at least one admissible process for the problem $(P) $, the latter admits an optimal solution.
\end{theorem}
\begin{proof} 
By hypothesis, there exists at least one admissible process, and the cost function $ g $ is bounded from below. Consequently, the infimum of $(P)$ is finite, ensuring the existence of a minimizing sequence of admissible processes $(u_k, \{x_k(\cdot,\omega) : \omega \in \Omega\})_k$. The sequence of controls $\{u_k\}_k$ is uniformly bounded, which allows the extraction of a weak-* convergent subsequence (unlabeled for simplicity) $ u_k \overset{*}{\rightharpoonup} \bar{u} $ in $ L^\infty(0, T; \mathbb{R}^m) $, with $\bar{u} \in \mathcal{U}_{\rm ad}[t_0,T]$. Similarly, the associated trajectories $\{x_k(\cdot,\omega)\}_k$ are uniformly bounded for each $\omega \in \Omega$. Since \eqref{Comp.Conv} holds, we can apply \cite[Theorem 23.2]{Clarke} to extract a further subsequence that converges uniformly to a limit function $\bar{x}(\cdot, \omega) \in W^{1,1}$.

An argument analogous to that in \cite[Theorem III.2]{AronnaMonteiroSierra2025CDC} for scalar-valued controls shows that the family $\{\bar{x}(\cdot, \omega) : \omega \in \Omega\}$ represents the trajectory associated with the control $\bar{u}$. Furthermore, the cost function of the sequence $\{x_k(\cdot,\omega)\}_k$ converges to the cost of the process $(\bar{u}, \{\bar{x}(\cdot,\omega) : \omega \in \Omega\})$, completing the proof.

\end{proof}
\begin{remark}
The conclusion in \cite[Theorem 23.2]{Clarke} also guarantees the existence of an admissible control $\hat{u}$ associated with the limit trajectory $\{\bar{x}(\cdot,\omega) : \omega \in \Omega\}$. However, no information is provided regarding the convergence of $u_k$ to $\hat{u}$ in any specific sense. Note that if the fields $f_1, \dots, f_m$ are linearly independent, then $\bar{u} = \hat{u}$ almost everywhere.

\end{remark}
\subsection{Necessary optimality conditions}
In this section, we verify that the hypotheses of the {\it smooth-case} version of the Pontryagin Maximum Principle given in Bettiol-Khalil \cite[Theorem 3.3]{Bettiol2019} are holding in our framework. We first introduce some necessary definitions and additional regularity assumptions.

\begin{definition}[$W^{1,1}$-local minimizer]
    \label{w11_minimizer}
    A process  $(\bar{u}, \{\bar{x}(\cdot,\omega): \omega\in \Omega\})$, is said to be an {\it $W^{1,1}$-local minimizer} for $(P) $ if there exists $\epsilon>0$ such that 
    \begin{equation*}
        \int_\Omega g(\bar{x}(T,\omega),\omega)d\mu(\omega)\leq
        \int_\Omega g(x(T,\omega),\omega)d\mu(\omega),
    \end{equation*}
    for all admissible processes $(u, \{x(\cdot,\omega): \omega\in \Omega\})$  such that
    \begin{equation*}
        \| \bar{x} (\cdot,\omega) - x(\cdot,\omega) \|_{W^{1,1}([0,T];\mathbb{R}^n)} \leq \epsilon \; , \; \forall \omega \in \supp (\mu).
    \end{equation*}
\end{definition}

\bigskip

We recall the assumptions and results established by Bettiol-Khalil in \cite{Bettiol2019}.  Given a $W^{1,1}$-local minimizer $(\bar{u}, \{\bar{x}(\cdot, \omega): \omega \in \Omega\})$ and $\delta > 0$, we shall refer to the following set of assumptions as \textbf{  (H1)}:

\begin{itemize}\label{A5}
   \item[(i)] There is a modulus of continuity $\theta_f:[0,\infty)\rightarrow[0,\infty)$ such that, for all $\omega,\omega_1,\omega_2\in\Omega,$ 
    \begin{gather*}
        \int_{t_0}^T \sup_{x\in\Bar{x}(t,\omega)+\delta \mathbb{B}, u\in U(t)} |f(x,u,\omega_1) - f(x,u,\omega_2)| dt \leq \theta_f(\rho_\Omega(\omega_1,\omega_2))
    \end{gather*} 
\item[(ii)] $g(\cdot,\omega)$ is differentiable on $\Bar{x}(T,\omega)+\delta \mathbb{B}$, for each $\omega\in\Omega,$  $\nabla_x g(\cdot,\omega)$ is continuous on $\Bar{x}(T,\omega)+\delta \mathbb{B},$ $\nabla_x g(x,\cdot)$ is continuous on $\Omega$ and there exists a positive constant $C_g$ such that $|\nabla_x g(\bar x(t,\omega),\omega)|\leq C_g$; 
\end{itemize}
\begin{itemize}\label{f.contdiff} 
   \item[(iii)] $x\mapsto f(t,x,u,\omega)$ is continuously differentiable on $\Bar{x}(t,\omega)+\delta \mathbb{B}$ for all $u\in U(t),$ $\omega\in\Omega$ a.e. $t\in[t_0,T]$, and $\omega\mapsto\nabla_x f(t,x,u,\omega)$ is continuous uniformly with respect to $(t,x,u).$   
  \item[(iv)]  We assume that the functions $f_i$ are twice continuously differentiable on $\Bar{x}(t,\omega)+\delta \mathbb{B}$ for all $\omega\in\Omega$  and a.e. $t\in[t_0,T]$ and there exists $C_f^k>0$ such that, for $i\in\{0,\dots,m\}, k=1,2,$
    \begin{equation*}
        \left|\frac{\partial^k f_i}{\partial x^k}(\bar{x}(t,\omega),\omega)\right|< C_{f_i}^k ,\quad \omega\in\supp(\mu),\text{ a.e. } t \in [0,T].
    \end{equation*}
\end{itemize}
\begin{theorem}{\cite[Theorem 3.3]{Bettiol2019}}
\label{pmp}
Let \textbf{  (H0)} and \textbf{  (H1)} hold and $(\bar{u}, \{\bar x(\cdot,\omega) : \omega \in \Omega \} )$ be a $W^{1,1}$-local minimizer for $(P)$. Then, there is a $I\times\mathcal B_\Omega$-measurable function $p : [0,T] \times \Omega \rightarrow \mathbb{R}^n$ such that $ p(\cdot,\omega) \in W^{1,1}([0,T], \mathbb{R}^n)$ for all $\omega \in\supp (\mu),$ and 
\begin{multline}\label{pmp.finite}
    \int_\Omega p(t,\omega)f(\bar x(t,\omega),\bar{u}(t),\omega)d\mu(\omega) = 
    \\ \max_{u \in {U(t)}}\int_\Omega p(t,\omega)f(\bar x(t,\omega),u,\omega)d\mu(\omega) \ \text{ a.e } \ t \in [0,T] ,
\end{multline}
\begin{multline}
    \label{pmp:adjoint}
    -\dot{p}(t,\omega) = [\nabla_x f(\bar{x}(t,\omega),\bar{u}(t),\omega)]^\top p(t,\omega),
\ \text{ a.e. } t\in [0,T],\text{ for all }\omega \in \mathrm{supp}(\mu),
\end{multline}

\begin{equation}
    -p(T,\omega) = \nabla_x g(\bar{x}(T,\omega),\omega),\text{ for all }\omega \in \mathrm{supp}(\mu).
\end{equation}
\end{theorem}
With the previous considerations, we verify that the above theorem holds. Condition \eqref{pmp.finite} will be applied in Section \ref{char} to characterize the singular arcs.
The next section deals with the more general case where there is uncertainty on the initial condition.

\section{Uncertain initial condition}\label{UIC}

In this section, we aim at showing necessary conditions for the problem $(P)$ with uncertain initial condition, {\it i.e.,} we suppose that $\varphi$ belongs to the space of measurable functions $\Vspace$, given by
$$
L^2(\mu,\Omega;\mathbb{R}^n):=\Big\{\varphi:\Omega\rightarrow \mathbb{R}^n : \int_{\Omega}|\varphi(\omega)|^2\,d\mu(\omega)<\infty\Big\},
$$
which is a Hilbert space endowed with the scalar product
$$
\langle \varphi,\psi\rangle:= \int_{\Omega}\varphi(\omega)\cdot \psi(\omega)\,d\mu(\omega) \quad \text{for} \quad \varphi,\psi \in L^2(\mu,\Omega;\mathbb{R}^n).
$$
 We will use the notation $
\|\cdot \|_{L^2}$ to refers to the norm associated with the above scalar product. Here, we redefine the dynamics as a map taking values in an infinite-dimensional space, more precisely, we define the dynamics   $$f: L^2(\mu,\Omega;\mathbb{R}^n)\times \mathbb{R}^m \rightarrow L^2(\mu,\Omega;\mathbb{R}^n)$$ as  
\begin{equation}\label{dyna.infinito}
\begin{aligned}
f(\varphi(\cdot),u)&:=
f_0(\varphi(\cdot),\cdot) + \sum_{i=1}^m f_i(\varphi(\cdot),\cdot)u_i.
\end{aligned}
\end{equation}
Condition \textbf{  (H0)} implies the existence and uniqueness of a solution of \eqref{trajectory}, as shown below.
\begin{lemma}
If \textbf{  (H0)} hold true, then, for any $(\varphi, u) \in  L^2(\mu,\Omega;\mathbb{R}^n) \times \mathcal{U}_{ad}[t_0,T],$ there exists a unique solution $x\in C([t_0,T]: L^2(\mu,\Omega;\mathbb{R}^n))$ of \eqref{trajectory}.
\end{lemma}
\begin{proof}
      From conditions \textbf{  (H0)}, we deduce directly that $f$ satisfies the following inequalities: 
\begin{equation}\label{H'}
\begin{aligned}
&||f(\varphi,u)||_{L^2}\leq C_{u}\left(
1+||\varphi||_{L^2}\right),\\
    &||f(\varphi,u)-f(\psi,u)||_{L^2}\leq K_{u}||\varphi-\psi||_{L^2}   
\end{aligned}
\end{equation}
where $$C_u=C\left(c_0+||u||_{L^{\infty}}\sum_{i=1}^{m}c_i\right) \ \text{and} \ K_u=\left(k_0+||u||_{L^{\infty}}\sum_{i=1}^{m}k_i\right),$$ for some positive constant $C.$ It is clear 
(see, {\it e.g.,} \cite[Ch.2, Proposition 5.3]{liYong}) that under conditions \eqref{H'}, there exists a unique solution $x\in C([t_0,T]: L^2(\mu,\Omega;\mathbb{R}^n))$ of \eqref{trajectory}.
\end{proof}
We will denote the set of trajectories of \eqref{trajectory} with initial data $\varphi\in L^2(\mu,\Omega;\mathbb{R}^n)$  as 
\begin{equation}\label{set.trajectory}
\begin{aligned}
    S(\varphi):= \left\{ x \in C([t_0,T]: L^2(\mu,\Omega;\mathbb{R}^n)) :  x  \ \text{solves  \eqref{trajectory} with }  u\in \mathcal{U}_{ad} \right\}.
   \end{aligned}
\end{equation}

\begin{remark}\label{F.CC.L2}
    Note that condition ${\bf  (H0)}$ and the affine structure of the dynamics imply that for all $\varphi\in \Vspace,$ the set
$f(\varphi,U(t))\subset\Vspace$ is compact and convex for {\it a.e.} $t\in [t_0,T].$
\end{remark}

From now on, until otherwise specified, we will use $x(t):=x(t,\cdot)$ to denote the solution of \eqref{trajectory}.

\subsection{Existence of optimal controls}

The existence of optimal trajectories becomes more challenging when the initial data is uncertain, mainly due to the lack of compactness of bounded sets of $\Vspace$. In the proof of Theorem \ref{Existence}, a key step in establishing the existence result 
relies on \cite[Theorem 23.2]{Clarke}, which ensures the sequential compactness of the set of admissible trajectories. However, latter result does not extend to infinite-dimensional spaces. Nevertheless, this challenge can be addressed under assumption \textbf{  (H0)}, combined with additional conditions imposed on the dynamics. One can derive an existence result in any of the four frameworks listed below by applying the methodologies outlined \cite{A.P.S,Frankowska1992} and \cite{liYong}:
\\
\\
\textbf{  1. Compact Semigroup}  (see \cite{Frankowska1992,liYong}). 
The infinite-dimensional ODE in \eqref{dynamics} verifies
\begin{equation*}
\textbf{(E1)} \quad
\begin{aligned} 
& \qquad \dot{x}(t,\omega) = A x(t,\omega) + \sum_{i=1}^m f_i(x(t,\omega),\omega) u_i(t), \\
& \text{where } A \text{ is the infinitesimal generator of a compact semigroup.}
\end{aligned}
\end{equation*}
\textbf{  2. Compactness Condition} (see \cite{Frankowska1992,A.P.S}).
\begin{equation*}
  \textbf{(E2)}  \begin{aligned}
      \quad &\text{For all} \ R>0, \text{there exists } \  K_{R}\subset \Vspace \ \text{ compact, such} \\&\text{that} \ f(\varphi,U(t))\subset K_R \ \ \text{for every} \quad (t,\varphi)\in [t
    _0,T]\times \mathbb{B}_{R}(0).
    \end{aligned}
    \end{equation*}
 \textbf{  3. Strong Invariance.} Let $K\subset \Vspace$ be a compact subset, define $G:=[0,T]\times K$ and consider the autonomous differential inclusion 
  \begin{equation}\label{dyna2}
\begin{cases}
    \dot{v}(t) \in \{1\}\times f(x(t),U(t)) & \text{a.e} \ t\in[t_0,T],  \\
 v(t_0) = (t_0,\varphi)\in G.
\end{cases}
\end{equation}  
   \begin{definition}
        We say that the pair $(\{1\}\times f,G)$ is {\it strongly invariant} if, for
all $(t_0,\varphi)\in G,$ all the trajectories of the differential inclusion \eqref{dyna2} remain in $G,$ {\it i.e.,} $(t,x(t))\in G$ for all $t\in 
[t_0,T].$
\end{definition}
\begin{definition}[Contingent Cone]
     For $v\in G ,$ define the {\it contingent (or Bouligand) cone} by
    $$
    T_{G}(v):=\left\{ v':\liminf_{h\rightarrow 0^+}{h^{-1}\text{dist}(v+hv',G)=0}\right\},
    $$
    where $\text{dist}(a,B):=\inf\left\{|a-b|:b\in B\right\}.$
\end{definition}
Consider the following condition on the dynamics:
$$
\textbf{(E3)}\hspace{0.4cm} \parbox{.8\textwidth}{  There exists a full-measure subset $I$ of $[t_0,T]$   such that, for all $(t,\varphi) \in I \times K,$ $
\{1\}\times f(\varphi,U(t))\subset T_{G}(t,\varphi).$
}
$$
From \cite[Theorem 2]{Donchev}, the following result holds.
\begin{lemma}\label{strong.inv}
Under condition $\textbf{ (H0)},$  the  property $\textbf{ (E3)}$  is equivalent to the strongly invariance of the system $(G,\{1\}\times f).$ 
\end{lemma}
\noindent\textbf{  4. Additional regularity.}
\begin{equation*}
  \textbf{(E4)}  \begin{aligned}
      \quad &  \text{Suppose that $\Omega \subset \mathbb{R}^d$ and that $\mu$ is the Lebesgue measure.}\\ 
      & \text{Let $\varphi \in H^1(\Omega; \mathbb{R}^n)$, and assume that the modulus of continuity } \\
      & \text{given in Lemma \ref{cont.traj}-(2) satisfies $\lim_{h \to 0^{+}} \theta_f(h)/h = 0$.}
    \end{aligned}
    \end{equation*}
If $x \in C([t_0,T]; \Vspace)$ is the solution of \eqref{trajectory} for some control $u \in \mathcal{U}_{ad}$, then, from \eqref{H'} and
 Gronwall's inequality, for all $t \in [t_0,T]$ and all $\omega_1, \omega_2 \in \Omega$, we have that
\begin{equation}\label{quotients}
    |x(t,\omega_1) - x(t,\omega_2)| \leq C |\varphi(\omega_1) - \varphi(\omega_2)| + C \theta_f(|\omega_1 - \omega_2|),
\end{equation}
for some positive constant $C.$ Moreover, from difference quotients (see \cite{evans2010}), it follows that $x(t, \cdot) \in H^1(\Omega; \mathbb{R}^n).$ By the Rellich–Kondrachov Theorem, the embedding $H^1(\Omega; \mathbb{R}^n) \hookrightarrow \Vspace $ is compact. With this in mind, consider a sequence of controls $\{u_k\}_k \subset \mathcal{U}_{ad} $ that converges to $\bar{u} $ in the weak*-topology of $L^\infty(t_0,T; \mathbb{R}^m) $. Then, for all $t \in [t_0,T] $, the corresponding sequence of trajectories $\{x_k(t, \cdot)\}_k \subset H^1(\Omega; \mathbb{R}^n) $ is bounded (by Lemma \ref{x.ineq} in Appendix \ref{AppA2} and \eqref{quotients}). 
Consequently, there exists a subsequence (using the same index) such that $$x_k(t, \cdot) \rightarrow x^*(t, \cdot) \ \text{strongly in} \ \Vspace$$ for some $x^*(t, \cdot)\in \Vspace.$ 
\\
\begin{theorem}[Existence of optimal control in the infinite-dimensional case] Suppose that \textbf{  (H0)} and \textbf{  (Ei)} hold for some $i\in \{1,2,3,4\},$ then, if there is at least one admissible process for the problem $(P) $, the latter admits an optimal solution.
\end{theorem}
\begin{proof}
If we assume condition  \textbf{  (E3)}, from Lemma \ref{strong.inv} it follows directly  the pre-compactness of $S(\varphi)$ in $C([t_0,T]:\Vspace)$. On the other hand, by assuming one of the conditions \textbf{  (E1),(E2),} or \textbf{  (E4)} above, along with the Arzelà-Ascoli Theorem, we obtain the pre-compactness of $S(\varphi).$ Furthermore, Remark \ref{F.CC.L2} and Filippov's Lemma
ensure the closedness of the set of trajectories, {\it i.e.,} we derive an infinite-dimensional version of \cite[Theorem 23.2]{Clarke}. Consequently, the existence of optimal trajectories is established by following an argument analogous to that used in the proof of Theorem \ref{Existence}. 
\end{proof}

\subsection{Necessary optimality conditions}
To establish the necessary optimality conditions for the problem $(P)$ with uncertain initial condition $\varphi \in \Vspace$, we will use technical lemmas whose proofs are included in Appendix \ref{AppA}. 

\begin{lemma}[Differentiability of the reduced cost]\label{FJlemma}
 Suppose that $g$ is differentiable with respect to $\bar x,$ and that $\nabla_xg(\cdot,\cdot)$ is continuous. Set $\delta>0.$ Then for all $\varphi\in \bar{x}(T,\cdot) + \delta\mathbb{B},$  the {\it reduced cost functional} $$\mathcal{J}(\varphi):= \int_\Omega g(\varphi(\omega),\omega)d\mu(\omega)$$ is Fréchet differentiable in $\varphi$, and the derivative of $\mathcal{J}$ in $\varphi$ is the functional $D\mathcal{J}(\varphi)\in\Vspace^*$ given by

\begin{equation}\label{FréchetJ}D\mathcal{J}(\varphi)\psi=
 \int_{\Omega}  \nabla_x g(\varphi(\omega),\omega)\cdot \psi(\omega) \,d\mu(\omega).
 \end{equation}
Note that, by the Riesz Representation Theorem, $\nabla_xg(\varphi(\cdot),\cdot) \in \Vspace $ is the unique representative of the derivative of $\mathcal{J}$. Thus, we denote the derivative as $D\mathcal{J}(\varphi) := \nabla_x g(\varphi).$

\end{lemma}
\begin{proof} See Appendix \ref{AppA2}.
\end{proof}
\begin{remark}
    It is straightforward to verify that under conditions \textbf{(H0)} and \textbf{(H1)}, for some control $u,$ the function defined in \eqref{dyna.infinito} is Fréchet differentiable, and its derivative in $\varphi\in \Vspace$ is the linear map $\frac{\partial f}{\partial x}(\varphi,u):\Vspace \rightarrow \Vspace $  given by {\small $$
    \frac{\partial f}{\partial x}(\varphi,u)\,\psi=\left(  \frac{\partial f_0}{\partial x}(\varphi(\cdot),\cdot) + \sum_{i=1}^m \frac{\partial f_i}{\partial x}(\varphi(\cdot),\cdot)u_i\right)\psi(\cdot)
    $$}
    where $\frac{\partial f}{\partial x}$ is the Jacobian of the function defined in \eqref{dyna.finito}. Additionally, 
    note that $\left(\frac{\partial f}{\partial x}(\varphi(\omega),u)\right)^*=\left(\frac{\partial f}{\partial x}(\varphi(\omega),u)\right)^\top$ for a.e. $\omega \in \Omega.$
\end{remark}
Let us use $(\bar{x},\bar{u})$ to denote 
  an optimal pair for problem $(P)$ with initial time $t_0$ and data $\varphi.$ Consider the linear problem
\begin{equation}\label{G(s,t)}
\left\{\begin{array}{l}
\frac{\partial G}{\partial s}(s, t)=\left(\frac{\partial f}{\partial x}(\bar{x}(s), \bar{u}(s))\right) G(s, t), \quad t_0\leq t \leq s \leq T, \\
G(t, t)=I.
\end{array}\right.
\end{equation}
 We denote by $G(s, t)$ the solution operator which is, in fact, the evolution operator (or fundamental matrix solution) of the linearization of \eqref{dynamics} along $(\bar{x}, \bar{u})$. In the following, $G^*(s, t)$ is the adjoint of $G(s, t)$.
 
We obtain the following first order necessary optimality condition, for which a part of its proof is based on results from \cite{Frankowska1992}.

\begin{theorem}\label{pmpInf}
     Assume that  \textbf{  (H0)} and \textbf{  (H1)} 
hold. Let $(\bar{x},\bar{u})$ be a {\it $W^{1,1}$-local minimizer} for $(P)$ and set the functional $q:= \nabla_x g(\bar{x}(T)).$  Then, the function $$p(t):=-G^*(T, t) q \in W^{1,1}([0,T]; \Vspace)$$ is the solution of the backward problem 
\begin{equation}\label{backward}
\left\{ \begin{array}{l}
 \displaystyle -\dot p(t)=\left(\frac{\partial f}{\partial x}(\bar{x}(t), \bar{u}(t))\right)^* p(t), \quad t \in\left[t_0, T\right] \\
-p(T)=q,
\end{array}\right.
\end{equation}
and satisfies the Maximum Principle
\begin{equation}\label{Max.pri}
    \langle p(t), f( \bar{x}(t), \bar{u}(t))\rangle=\max _{u \in  U(t)}\langle p(t), f( \bar{x}(t), u)\rangle
    \end{equation}
for a.e. $t \in\left[t_0, T\right].$
\end{theorem}

\begin{proof}
The condition \eqref{backward} 
can be proven as in \cite[Theorem 3.1]{Frankowska1992}. The proof of inequality \eqref{Max.pri} is slightly different in our case since we are working within a purely nonlinear framework, therefore, we do not rely on semigroup theory. Instead, we use the regularity conditions imposed on the dynamic functions. 

Let us denote the set of Lebesgue points of the function $f( \bar{x}(\cdot), \bar{u}(\cdot))$ by 
    $$\mathcal{L}=\left\{ t \in[t_0, T] :\lim _{h \rightarrow 0^+} \frac{1}{h} \int_{t-h}^{t+h}|| f( \bar{x}(s), \bar{u}(s))-f( \bar{x}(t), \bar{u}(t)) ||_{L^2} ds=0\right\}.$$
We know that $\mathcal{L}$ has full measure in $[t_0,T].$ Fix $t\in\mathcal{L}$ and $v\in U(t),$ for sufficiently small $h>0$, we define the {\it needle perturbation} of $\bar{u}$ by 
    $$u_h(s)= \begin{cases}\bar{u}(s) & \text { if } s \in\left[t_0, T\right] \backslash[t-h, t], \\ v & \text { if } s \in[t-h, t]\end{cases}$$
    and let $x_h(\cdot)$ be the solution of \eqref{dynamics} with initial data $\varphi$ in time $t_0$ and control $u_h$. Let us consider the following linear problem    
    \begin{equation}\label{linearw}
    \left\{\begin{array}{l}w^{\prime}(s)=\frac{\partial f}{\partial x}( \bar{x}(s), \bar{u}(s)) w(s), \quad t \leq s \leq T, \\ \\w(t)=f( \bar{x}(t), v)-f(\bar{x}(t), \bar{u}(t)) .\end{array}\right.
    \end{equation}
Since $f$ is Fréchet differentiable with respect to $x,$ then, from \eqref{H'} the map $s\rightarrow\frac{\partial f}{\partial x}(\bar{x}(s), \bar{u}(s))$ is bounded, thus, \eqref{linearw} has a unique solution (see \cite{Pazy}) and from \eqref{G(s,t)} we have that 
\begin{equation}\label{w(s)}
w(s)=G(s,t)[f( \bar{x}(t), v)-f( \bar{x}(t), \bar{u}(t))] \ \text{ for all} \ s\in[t,T]. 
\end{equation}
First, we prove that $\left(x_h-\bar{x}\right)/h$ converges to $w$ as $h \rightarrow 0^+$, uniformly in $[t, T]$. Since $x_h(s)=\bar{x}(s)$ for all $s\in [t_0, t-h],$ then for all $s \in(t-h, t]$ it follows that
$$
\begin{aligned}x_h(s)-\bar{x}(s) & =\int_{t-h}^s \left[f\left( x_h(r), v\right)-f( \bar{x}(r), \bar{u}(r))\right] d r \\& =\int_{t-h}^s \left[f\left( x_h(r), v\right)-f( \bar{x}(r), v)\right]dr\\ & \hspace{3cm} +\int_{t-h}^s \left[f\left( \bar{x}(r), v\right)-f( \bar{x}(r), \bar{u}(r))\right]dr.
\end{aligned}
$$
Thus, from \eqref{H'} we have that 
$$
\begin{aligned}
\bigg|\bigg|& x_h(s)-\bar{x}(s) - \int_{t-h}^s \left[f\left(\bar{x}(r), v\right)-f( \bar{x}(r), \bar{u}(r))\right]dr\bigg|\bigg|_{L^2}
\\
&\leq \int_{t-h}^{s}K_u||x_h(r)-\bar{x}(r)||_{L^2} dr
\\
&\leq K_u\int_{t-h}^{s}\left|\left|x_h(r)-\bar{x}(r)-\int_{t-h}^r \left[f\left( \bar{x}(\sigma), v\right)-f( \bar{x}(\sigma), \bar{u}(\sigma))\right]d\sigma\right|\right|_{L^2} dr
\\
&  \hspace{2.9cm}+K_u\int_{t-h}^s\int_{t-h}^r \left|\left| f\left( \bar{x}(\sigma), v\right)-f(\bar{x}(\sigma), \bar{u}(\sigma))\right|\right|_{L^2}d\sigma\,dr.
\end{aligned}
$$
From \eqref{H'} and Lemma \ref{x.ineq} in Appendix \ref{AppA2}, it follows that
{\small $$
\int_{t-h}^s\int_{t-h}^r \left|\left| f\left(\bar{x}(\sigma), v\right)\right|\right|_{L^2}d\sigma dr \leq \int_{t-h}^t\int_{t-h}^t \left|\left| f\left( \bar{x}(\sigma), v\right)\right|\right|_{L^2}d\sigma dr \leq C_0h^2
$$}
and, similarly 
{\small $$
\int_{t-h}^s\int_{t-h}^r \left|\left| f\left( \bar{x}(\sigma), \bar{u}(\sigma)\right)\right|\right|_{L^2}d\sigma dr \leq C_1h^2,
$$}
for some positive constants $C_0$ and $C_1.$
Then,
$$
\begin{aligned}
&\left|\left| x_h(s)-\bar{x}(s) - \int_{t-h}^s \left[f\left( \bar{x}(r), v\right)-f(\bar{x}(r), \bar{u}(r))\right]dr\right|\right|_{L^2}
\\
&\leq  \int_{t-h}^{s}K_u\left|\left|x_h(r)-\bar{x}(r)-\int_{t-h}^r \left[f\left( \bar{x}(\sigma), v\right)-f( \bar{x}(\sigma), \bar{u}(\sigma))\right]d\sigma\right|\right|_{L^2} dr
\\
& \hspace{1cm}+ \bar{o}(h)
\end{aligned}
$$
where $\displaystyle \lim_{h\rightarrow 0^+}\bar{o}(h)/h=0.$ Applying Gronwall's inequality above, we obtain that
\begin{equation}\label{oh}
    \left|\left| x_h(s)-\bar{x}(s) - \int_{t-h}^s \left[f\left( \bar{x}(r), v\right)-f( \bar{x}(r), \bar{u}(r))\right]dr\right|\right|_{L^2}\leq o(h)
\end{equation}
for $s\in(t-h,t],$ where $\displaystyle \lim_{h\rightarrow 0^+}o(h)/h=0.$ On the other hand 
\small{
\begin{equation*}
    \begin{aligned}
 \left|\left|\int_{t-h}^t \left[f\left( \bar{x}(r), v\right)-f(\bar{x}(r), \bar{u}(r))\right]dr -h[f( \bar{x}(t), v)-f( \bar{x}(t), \bar{u}(t))]  \right|\right|_{L^2} 
    \end{aligned}
\end{equation*}}
\normalsize 
\begin{equation*}
    \begin{aligned}
    &\leq \int_{t-h}^t\left|\left| f\left( \bar{x}(r), v\right)-f( \bar{x}(t), v) \right|\right|_{L^2}\,dr\\
&\hspace{3.5cm}+\int_{t-h}^t\left|\left| f\left( \bar{x}(r), \bar{u}(r)\right)-f(\bar{x}(t), \bar{u}(t)) \right|\right|_{L^2}\,dr.
    \end{aligned}
\end{equation*}
From Lemma \ref{filipschitz} and since $t\in\mathcal{L}$ we have, respectively, that
$$
 \int_{t-h}^t\left|\left| f\left( \bar{x}(r), v\right)-f(\bar{x}(t), v) \right|\right|_{L^2}\,dr \leq C_2h^2
$$
and 
$$
\int_{t-h}^t\left|\left| f\left( \bar{x}(r), \bar{u}(r)\right)-f( \bar{x}(t), \bar{u}(t)) \right|\right|_{L^2}\,dr \leq o'(h)
$$
where $\displaystyle \lim_{h\rightarrow 0^+}o'(h)/h=0.$ The latter two inequalities, together with \eqref{oh}, lead to conclude that 
\begin{equation}\label{xh}
    x_h(t)=\bar{x}(t)+h[f( \bar{x}(t), v)-f( \bar{x}(t), \bar{u}(t))]+o(h),
\end{equation}
where $\displaystyle \lim_{h\rightarrow 0^+}o(h)/h=0.$ Finally, let $s\in[t, T].$ Then \eqref{xh} and the differentiability of solutions with respect to initial data imply that 
\begin{equation}\label{xh(s)}
||x_h(s)-\bar{x}(s)-hw(s)||_{L^2}\leq o(h), \ \text{for all} \ s\in [t,T].
\end{equation}
Since $\bar{x}$ is optimal, $\mathcal{J}(x_h(T))\geq\mathcal{J}(\bar{x}(T)),$ thus, from \eqref{xh(s)} and Lemma \ref{FJlemma}
\begin{equation}
\begin{aligned}
0&\leq \lim_{h\rightarrow0^+}\frac{\mathcal{J}(\bar{x}(T)+hw(T))-\mathcal{J}(\bar{x}(T))}{h}\\
&=\langle \nabla_x g(\bar{x}(T)),w(T) \rangle
=\langle q,w(T) \rangle,
\end{aligned}
\end{equation}
from \eqref{w(s)} we deduce that
$$
0\leq\langle\ q,G(T,t)[ f( \bar{x}(t), v)- f( \bar{x}(t), \bar{u}(t))]\rangle.
$$
Define $p(t)=-G^*(T, t) q.$ Thus,
\begin{equation*}
    \langle p(t), f( \bar{x}(t), v)\rangle \leq\langle p(t), f( \bar{x}(t), \bar{u}(t))\rangle
    \end{equation*}
    and \eqref{Max.pri} follows as $v\in U(t)$ is arbitrary. 
\end{proof}
\begin{remark}
In \cite{Frankowska1992} it is also proved that $p$ satisfies the co-state inclusion
\begin{equation*}
p(t) \in D_x^{+} V(t, \bar{x}(t)), \quad \forall t \in\left[t_0, T\right] 
\end{equation*}
and that for all $t \in \mathcal{L}$,
\begin{equation}\label{HD+}
    (H(t, \bar{x}(t), p(t)),-p(t)) \in D^{+} V(t, \bar{x}(t)),\end{equation}
where $V:[0, T] \times \Vspace \rightarrow \mathbb{R}\, \cup\,\{ \pm \infty\}$ is the value function associated to the problem $(P)$ defined as follows:
 \begin{equation*}\label{vf2}
 V \left(s, \varphi\right):=\inf_{u\in\mathcal{U}_{ad}} \Bigg\{ \int_{\Omega}g(x(T,\omega),\omega)\,d\mu(\omega): x \  \text{solves}\ \eqref{trajectory} \ \text{with } t_0=s \Bigg\},
 \end{equation*}
 and for $(t, \varphi) \in A$, where $A $ is an open subset of $[0,T] \times L^{2}(\mu, \Omega; \mathbb{R}^{n})$, the superdifferential of $V $ at $(t, \varphi)$, denoted by $D^+ V(t, \varphi) $, is defined as:
\begin{equation}\label{superdifferential}
\begin{aligned}
&D^{+}V(t,\varphi):=\Big\{(p_t, p_\varphi)\in \mathbb{R}\times \Vspace^{*}: \\
&\limsup_{(s,\varphi')\rightarrow (t,\varphi), \,(s,\varphi')\in A} \frac{V(s,\varphi')-V(t,\varphi)- \langle p_\varphi, \varphi'-\varphi\rangle -p_t(s-t)}{|s-t|+||\varphi'-\varphi||_{L^{2}}}\leq 0\Big\}.
\end{aligned}
\end{equation}
In \eqref{HD+}, $H:[0,T]\times \Vspace \times \Vspace^* \rightarrow \mathbb{R}$ is the Hamiltonian of the controlled system \eqref{dynamics}, defined by 
$$H(t, \varphi, p(t))=\sup _{u \in U(t)}\langle p(t), f(\varphi, u)\rangle.$$

Other works addressing necessary conditions with alternative functionals $J$ in
problem $(P)$ can be found, for example, in \cite{BonalliBonnet2023}. 
In this direction, one may consider alternative risk measures in problem $(P)$; 
for example, for some $\alpha \in (0,1]$ one can employ the
Conditional Value-at-Risk (CVaR) risk measure and consider the problem
$$
\min_{u \in \mathcal{U}_{\rm ad},\, s \in \mathbb{R}}
\left[
  s + \frac{1}{\alpha}
  \int_{\Omega} \max\{ g(x(T,\omega),\omega) - s , 0 \} \, d\mu(\omega)
\right].
$$
This formulation represents a ``here-and-now" approach,
whereas our setting follows a ``wait-and-see" framework. 

\end{remark}


\section{Singular control characterization}\label{char}

In this section, we characterize the singular arcs of optimal controls of problem $(P).$ First, the scalar control case is analyzed, and a formula for the singular arcs is derived. Next, we proceed to investigate the case of vector-valued controls under the assumption that the system dynamics are commutative.
The following characterizations hold for both cases, known and unknown initial conditions, once they do not depend on the initial state of the system.

\subsection{Scalar control variable}

Throughout this subsection, we assume that the control variable is scalar, {\it i.e.,} $m=1.$ We define our control set as $U(t) := [u_{\min}, u_{\max}]$ for real numbers $u_{\min}< u_{\max},$ for all $t\in [0,T].$ The ODE in \eqref{dynamics} reads  
\begin{equation}\label{scalar.dyn}
\dot x(t,\omega) = f_0(x(t,\omega), \omega) + {u(t) f_1(x(t,\omega), \omega)}.   
\end{equation}
\begin{theorem}
[Characterization of the optimal control for the scalar control-affine case]
\label{sing_control_scalar}
Let assumptions \textbf{  (H0)} and \textbf{  (H1)} hold, and let $(\bar u, \{\bar{x}(\cdot,\omega) : \omega \in \Omega \})$ be a $W^ {1,1}$-minimizer for problem $(P)$, and consider  $p$ as in Theorem \ref{pmp} or Theorem \ref{pmpInf} (for the finite- and infinite-dimensional cases, respectively). Then, the optimal control $\bar u$ satisfies the conditions
\begin{equation}
\label{controls}
  \bar u(t) =
    \begin{cases}
      u_{\max} & \text{if } \Psi(t) > 0, \\
      u_{\min} & \text{if } \Psi(t) < 0, \\
      \mathrm{ singular } & \text{if } \Psi(t) = 0,
    \end{cases}       
\end{equation}
almost everywhere in $[0,T],$ where $\Psi$ is the {\it switching function} given by
\begin{equation}
\label{switching_func}
\Psi(t) := \int_\Omega p(t,\omega)f_1(\bar{x}(t,\omega),\omega) d\mu(\omega).
\end{equation}
Moreover, over any open interval $(a,b)$ on  which $\Psi = 0$, the control $\bar u$ satisfies
\begin{equation}
    \label{control_formula}
     {\int_\Omega p[f_0,[f_0,f_1]] \,d\mu(\omega) + \bar u(t)\int_\Omega p[f_1,[f_0,f_1]] \,d\mu(\omega)=0.}
\end{equation}
\end{theorem}
\begin{proof}
From the maximum conditions \eqref{pmp.finite} and \eqref{Max.pri} of the Pontryagin Maximum Principle, stated in Theorem \ref{pmp} and Theorem \ref{pmpInf}, respectively, we obtain, that for a.e. $t \in [0,T]$,

\begin{equation}
    \bar{u}(t) \in \argmax_{u \in [u_{\min},u_{\max}]} \int_\Omega p(t,\omega) \big[f_0(\bar x(t,\omega),\omega) +  {u} f_1(\bar x(t,\omega),\omega)\big] \, d\mu(\omega),
\end{equation}
which gives
\begin{equation}
    \bar{u}(t) \in \argmax_{u \in [u_{\min},u_{\max}]}  u\,\Psi(t).
\end{equation}
Therefore, if $\Psi(t) \neq 0$, we obtain one of the two cases described in the first two rows of equation \eqref{controls}. {Assume now that} $\Psi(t) = 0$ {over an open interval $(a,b)\subset [0,T]$.} By  assumptions \textbf{(H0)}, \textbf{(H1)} and Lemma \ref{lemma_xp} in Appendix \ref{AppA}, {one has that $d (p f_1)/dt$} is uniformly bounded (see \cite{AronnaMonteiroSierra2025CDC}). Consequently, differentiation under the integral sign {in the expression of $\Psi$ in \eqref{switching_func}} is well-defined and we obtain that  
\begin{equation}
\label{first_derivative}
\dot \Psi  =
{\int_\Omega \frac{d (p f_1)}{dt}  \,d\mu(\omega)=}
\int_{\Omega}p(f_1' f_0 - f_0' f_1)\,d\mu(\omega)=\int_\Omega p[f_1,f_0]\,d\mu(\omega)=0,
\end{equation} 
where we have used the chain rule together with the differential equations \eqref{pmp:adjoint} (or \eqref{backward}) and \eqref{scalar.dyn}.
Now, one can see that
$$
 \frac{d (p[f_1,f_0])}{dt} = p[f_0,[f_0,f_1]] + up[f_1,[f_0,f_1]].
$$
From \textbf{(H1)}-(iv) together with Lemma \ref{lemma_xp}, we obtain that $ \frac{d (p[f_1,f_0])}{dt} $ is uniformly bounded.
Once again, the derivation under the integral sign {in \eqref{first_derivative}} is well defined, leading to 
\begin{equation}
\label{second_derivative}
    \ddot{\Psi} = \int_\Omega p[f_0,[f_0,f_1]] + up[f_1,[f_0,f_1]]\,d\mu(\omega) = 0.
\end{equation}
Then, formula \eqref{control_formula} holds.
\end{proof}

\subsection{Vector control variable and commutative dynamics}

In this subsection, we consider the case for the vector-valued control setting with {\it commutative} vector fields $f_i,$ for $i=1,\dots,m$.  This is, throughout this subsection we assume that
\begin{equation}
\label{vector field commutativity}
  [f_i,f_j] \equiv 0 \  \ \text{for } \  i,j =1,\dots,m,
\end{equation}
and 
$$
U(t):=\prod_{i=1}^{m}[u_{\min}^i,u_{\max}^i],
$$
{for  real numbers $u_{\min}^i < u_{\max}^i,$ for $i=1,\dots,m.$}

In order to present the characterization result, we first provide the definition for the set of singular indexes and other involved elements. {Throughout the remainder of the section, we fix a process $(\bar u, \{\bar{x}(\cdot,\omega) : \omega \in \Omega \})$, and we assume that whenever the argument of the function is omitted, it is evaluated at this process.}

\begin{definition}[Set of singular indexes]
For the problem $(P),$ we define the set of singular indexes as 
 \begin{equation}
    S(t) := \left\{ i\in \{1,\dots m\}: u_{\min,i} < \bar u_i (t) <u_{\max,i} \right\},
\end{equation}
{\it i.e.,} the set of indices corresponding to the coordinates of singular controls at a given time $t\in[t_0,T].$ Note that $S(t)$ is defined up to a zero-measure set.
\end{definition}
\begin{definition}[Set of bang indexes]
    For $t\in [t_0,T]$, we define the {\it sets of bang indexes} as
\begin{equation*}
    B_{\min}(t) := \left\{ i\in \{1,\dots m\}: \Psi_i (t)<0 \right\},
\end{equation*}
\begin{equation*}
    B_{\max}(t) := \left\{ i\in \{1,\dots m\}: \Psi_i (t)>0 \right\},
\end{equation*}
{where, for $i=1,\dots,m,$
we set
\(
\Psi_i(t) := \int_\Omega p(t,\omega)f_i(\bar{x}(t,\omega),\omega) d\mu(\omega).
\)
for the switching function corresponding to the control component $i.$}
\end{definition}
\begin{definition}
    Let $f_1,\dots,f_m$ be the vector fields of problem $(P)$ and set 
\begin{equation}\label{Wij}
W_{ij} := \int_\Omega p [f_i,[f_0,f_j]] d\mu(\omega), \text{ for } \,  i,j = 0,\dots,m, 
\end{equation}
We consider the following time-dependent matrix of singular indexes $W_S(t):=[W_{ij}]_{i,j\in S(t)},$ which is involved in the characterization of the singular controls. 
\end{definition}

We now present the theorem that characterizes the singular controls .

\begin{theorem}
\label{Thmfeedbackinfinite}
Let assumptions \textbf{  (H0)} and \textbf{  (H1)} hold, and let $(\bar u, \{\bar{x}(\cdot,\omega) : \omega \in \Omega \})$ be a $W^ {1,1}$-minimizer for problem $(P).$ Consider $p$ as in Theorem \ref{pmp} or Theorem \ref{pmpInf}. Additionally, suppose that the commutativity condition \eqref{vector field commutativity} holds. Then, on any open interval $(a,b)$ in which $S(t),$ $B_{\max}(t)$ and $B_{\min}(t)$ remain constant (and nonempty), one has:
\begin{equation}
\label{mult dim vector}
  { W_S \bar{u}_S + b_S =0, }
\end{equation}
where $\bar u_S$ is the vector of singular components of the control $\bar u$ and, for $i\in S(t),$ $
b_i :=  W_{0i} +  \sum_{j \in B_{\min}(t) } W_{ij} u_{\min,j}
    +\sum_{j \in B_{\max}(t) } W_{ij} u_{\max,j}.
$
\end{theorem}
\begin{proof}
Analogously to Theorem \ref{sing_control_scalar}, the coordinates of the optimal control $\bar u$ satisfies the conditions
\begin{equation*}
  \bar u_i(t) =
    \begin{cases}
      u_{\max}^i & \text{if } \Psi_i(t) > 0, \\
      u_{\min}^i & \text{if } \Psi_i(t) < 0, \\
      \mathrm{ singular } & \text{if } \Psi_i(t) = 0,
    \end{cases}       
\end{equation*}
where the switching function associated to the coordinate $\bar u_i$ is defined as 
$$
\Psi_i := \int_\Omega p f_i d\mu(\omega).
$$
Let $(a,b)$ be an interval as in the statement and  $i\in \{1,\dots,m\}$ in $S(t)$ for $t\in (a,b).$ From condition \textbf{(H1)}-(iv) and Lemma \ref{lemma_xp}, we conclude that {$dp f_i/dt$} is uniformly bounded, then
$$
\dot \Psi_i= \int_\Omega p[f_0,f_i]d\mu(\omega)=0.
$$
Arguing as in Theorem \ref{sing_control_scalar}, we obtain that $d p[f_0,f_i]/dt$
is uniformly bounded; consequently, one gets that
\begin{equation*}
  \ddot \Psi_i=  \int_\Omega p\left([f_0,[f_0,f_i]] + \sum_{j=1}^m u_j [f_j,[f_0,f_i]]\right) d\mu(\omega)=0,
\end{equation*}
or equivalently,
\begin{equation}\label{ddotPsii}
    \ddot \Psi_i= W_{0i}+\sum_{j=1}^mW_{ij}u_j=0.
\end{equation}
Then, the expression \eqref{ddotPsii} can be rewritten as
\begin{equation}
\label{sing arc multi}
    b_i + \sum_{j \in S(t) } W_{ij} \bar{u}_j  = 0, \,\, \text{for } i \in S(t).
\end{equation}
The proof follows.
\end{proof}
\begin{remark}
  Supposing that the matrix $W$ is definite,  $W_S$ is a principal submatrix and thus non-singular. Therefore, the singular controls at time  $t$  can be expressed as
\begin{equation*}
    \bar{u}_S = -W^{-1}_S b_S.
\end{equation*}
\end{remark}
\section{Numerical examples}\label{NUM_EXMP}
This section shows numerical simulations that validate the theoretical results derived in previous sections. We begin with a discussion on convergence properties before conducting numerical experiments on two examples.

\subsection{Numerical scheme}
\label{numerical scheme}

In order to numerically solve the problem $(P)$, we use the {\it sample average approximation.}
Specifically, at each iteration $k,$ a finite random set $\Omega_k=\{\omega_i^k\}_{i=1}^k$  
 of independent and identically $\mu$-distributed samplings is selected from the parameter space $\Omega$ and the following 
a classical (finite-dimensional) optimal control problem is considered:
\begin{equation*}
(P_k) \ \ \parbox{.9\textwidth}{
\begin{align}
&\text{minimize} \quad  J_{k}[u(\cdot)] := \frac{1}{k}\sum_{i=1}^{k}g({x}(T, \omega_i^k),\omega_i^k), \nonumber \\
&\text{s.t.} \nonumber \\
& \begin{cases}
\dot{x}(t,\omega_{i}^k) = f_0({x}(t, \omega_{i}^k),\omega_{i}^k) + \displaystyle \sum_{j=1}^m f_j({x}(t, \omega_{i}^k),\omega_{i}^k)u_j(t) , \quad i=1,\dots,k, \nonumber  \\
x(0,\omega_{i}^k)=\varphi(\omega_{i}^k), \quad i=1,\dots,k. \nonumber 
\end{cases}
\end{align}}\nonumber
\end{equation*}
Note that the state dimension in $(P_k)$ is $kn,$ $\mathbb{R}^n$ being the state space of the original problem.
In the case of known initial condition $x_0$, one sets  $\varphi(\omega_i^k) = x_0$, for $i=1,\dots,k.$  
Note that, since $ {\mathcal{U}}_{\rm ad}\subset L^2(0,T;\mathbb{R}^m),$ which is a complete, separable metric space with respect to the $L^2$-topology. Consequently, one can apply the numerical scheme developed in \cite{Phelps2016}, where the authors employ an extension of the {\it strong law of large numbers} (SLLN) to {\it random
lower semicontinuous} functions. Recall that the SLLN is the theoretical backbone of the Monte Carlo method. This numerical scheme estimates the objective functional $J$ defined in problem $(P)$ using the sample mean $J_k$ defined in  $(P_k).$ The convergence properties of problems  $(P_k)$ are analyzed using the concept of {\it epi-convergence} of the objective functionals. This notion of convergence offers a natural framework for studying the approximation of optimization problems, as it facilitates the discussion of the convergence of both their minimizers and infima.

\subsection{Comments on convergence} 
To apply the approximation scheme developed in \cite{Phelps2016} we need to introduce an additional condition on function $g$: 
$$
\textbf{  (G)}\hspace{0.4cm} \parbox{.9\textwidth}{
For each $y \in \mathbb{R}^n,$ there exists an open neighborhood $N_0$ of $y$ in $\mathbb{R}^n$ and an integrable function $\alpha_0: \Omega \rightarrow(-\infty, \infty)$ such that, for almost all $\omega \in \Omega,$ the inequality $
g(x, \omega) \geq \alpha_0(\omega)
$
holds for all $x \in N_0$.
}
$$
\begin{definition}[Epi-convergence] Let $(\mathcal{U}, d)$ be a complete,  separable  metric space, and consider a sequence of lower semicontinuous functions $\{F_k\},$ with
$$
F_k: \mathcal{U}\rightarrow (-\infty, \infty].
$$
The sequence $ F_k[\cdot] $ is said to {\it epi-converge} to $ F_0[\cdot] $, if the following two conditions are satisfied for every $ u_0 \in \mathcal{U} $:
\begin{enumerate}
    \item[(1)] $\liminf_{k\to\infty} F_k\left[u_k\right] \geq F_0\left[u_0\right]$, whenever $ u_k \rightarrow u_0 $.
    \item[(2)] $\lim_{k\to\infty} F_k\left[v_k\right] = F_0\left[u_0\right]$, for at least one sequence $ v_k \rightarrow u_0 $.
\end{enumerate}
\end{definition}
\begin{remark}
    The Lipschitz continuity of the function $g(\cdot,\omega),$ for all $\omega\in\Omega$ and Lemma \ref{cont.traj}-{\it (iii)} imply that the mappings ${\mathcal{U}}_{\rm ad}\ni u\mapsto J_k[u]$ are Lipschitz continuous, for every $k$.
\end{remark} 

\begin{theorem}\cite{Artstein1995,Phelps2016}\label{num_conv}
    Let $\omega_1,\dots,\omega_k,  $ be a sequence of independent and identically $\mu$-distributed drawings from $\Omega.$ Under conditions \textbf{  (H0)}-(iii) and \textbf{  (G)}, as $k\rightarrow\infty,$ the sequence $J_k$  epi-converges a.e. (or almost surely) to J.
\end{theorem}
The following result establishes a valuable convergence property for the approximate minimizers of problems  $(P_k),$ when the sequence of approximating objective functionals $J_k$ satisfies the condition of epiconvergence.

\begin{proposition}\cite{Attouch, Phelps2016}\label{inf_conv}
     Consider the sequence of Lipschitz continuous functions $J_k: {\mathcal{U}}_{\rm ad} \rightarrow \mathbb{R}$ such that $J_k$ epiconverges to $J$. If $\left\{u_k\right\}_{k \in \mathbb{N}} \subset {\mathcal{U}}_{\rm ad}$ is a sequence of global minimizers of $J_k$, and $\hat{u}$ is any accumulation point of this sequence (along a subsequence $\{u_{k_i}\}_{i\in \mathbb{N}}$ ), then $\hat{u}$ is a global minimizer of $J$ and $$\lim _{i\rightarrow \infty} \inf _{u \in {\mathcal{U}}_{\rm ad}} J_{k_i}\left[u\right]= \inf _{u \in {\mathcal{U}}_{\rm ad}} J[u].$$
\end{proposition}

\subsection{Example 1: Optimizing fishing strategies }

We now use a toy model to illustrate the results of Theorem \ref{num_conv} and Proposition \ref{inf_conv}.
We consider a modification of the optimal fishing problem presented in \cite{SoledadAronna2013}. We let $x(t,\omega)$ represent the size of the halibut fish population at time $t$, associated with $\omega$. 
The control $u(t)\in \mathbb R$ is the fishing effort per unit of time and the objective is to maximize the average fishing revenue over a fixed time interval $[0,T]$.
We obtain the following model:

\begin{equation}\label{fishing_original}
\begin{aligned}
&\displaystyle \text{maximize } J[u(\cdot)]=\int_\Omega \int_0^T \bigg(E-c/x(t,\omega) \bigg) u(t) U_{\text{max}} \,dt \, d\mu(\omega) \\
&\text{s.t}
\\
  &  \begin{cases}
       \dot x(t,\omega) = r(\omega) x(t,\omega)\left(1 - \frac{x(t,\omega)}{k(\omega)}\right) - u(t) U_{\text{max}}\quad \text{ a.e. on }[0,T],\\
        0\leq u(t) \leq 1\quad \text{ a.e. on }[0,T],\\
        x(0,\omega ) = \varphi(\omega).
    \end{cases}
    \end{aligned}
\end{equation}
with fixed parameters $T = 10$ for the final time, $E = 1$ for the selling cost, $c=17.5$ is the coefficient associated to the fishing cost, $U_{\max} = 20$ establishes the upper limit for the fishing effort, and we consider the following stochastic parameters: $\varphi \sim \mathcal{TN}(70,5,40,90)$ for the uncertain initial condition, $r \sim \mathcal{TN} (0.71, 0.05, 0.1, 1)$ the reproduction rate and $k \sim \mathcal{TN} (80.5, 10, 65, 95)$ the carrying capacity, where $\mathcal{TN}(\mu, \sigma, a, b)$ denotes the {\it truncated normal,} with mean $\mu$, standard deviation $\sigma$ in the interval $[a,b],$ as given in {\it e.g.} \cite{Kroese2011}. 
Recall that, since the probability density function of the truncated normal has an integrable square, it can be associated with an $L^2$-function.

After introducing an additional state variable $z,$ we transform \eqref{fishing_original} into the Mayer problem:
\begin{equation}
    \label{fishing_mayer}
   \begin{aligned}
    &\displaystyle \text{maximize   } J[u]=\int_\Omega  z(T,\omega) \, d\mu(\omega)\\
    &\text{s.t}\\
   &\begin{cases}
       \dot x(t,\omega) =
       rx(t,\omega)\left(1 - \frac{x(t,\omega)}{k}\right) - u(t) U_{\text{max}},\quad
       \text{ a.e. on }[0,T], \\
        \dot z(t,\omega)=\left(E-c/x(t,\omega) \right) u(t) U_{\text{max}}\quad
       \text{ a.e. on }[0,T], \\
        0\leq u(t) \leq 1\quad \text{ a.e. on }[0,T], \\
        x(0,\omega) = \varphi(\omega), \\
        z(0,\omega)=0.
    \end{cases}
    \end{aligned}
\end{equation}
Following Theorem \ref{Thmfeedbackinfinite}, the singular arc, if it exists, is given by, omitting arguments,
{\small \begin{equation*}
\begin{aligned}
u(t) = \displaystyle  -\frac{\displaystyle  \int_\Omega r^2p_z\left(\frac{1}{k} -\frac{1}{x}\right)+p_x\left( \frac{2x}{k}\left(1-\frac{x}{k}\right)-1\right) d\mu(\omega)}{2 U_{\max} \displaystyle \int_\Omega \frac{r}{k} \left(p_x + \frac{c}{x} p_z \right) d\mu(\omega)},
\end{aligned}
\end{equation*}}
where $p=(p_x,p_z)^T$ is the adjoint state.
We approximate the solution of \eqref{fishing_mayer} by solving the discretized version $(P_k)$ given in Subsection \ref{numerical scheme}, for sufficiently large values of $k$.
In order to solve  $(P_k)$ we use the GEKKO Library \cite{beal2018gekko}. For $k = 20$, we obtain 
$$
\displaystyle\frac{|J^{19} - J^{20}|}{|J^{19}|}=8.3\times 10^{-4},\qquad 
\frac{\|u^{19} - u^{20}\| }{\|u^{19}\|}=8.4\times10^{-4}.
$$
Figure \ref{fig:control_fishing} exhibits the control in two ways: the gray dotted function is the control approximated by using the sequence of problems (P$_k$) and the red one is the control computed by formula \eqref{control_formula}.
Figures \ref{fig:fishing trajectories} and \ref{fig:fishing control arcs} illustrate the values of $x^k(t)$ and $u^k(t),$ for each $k =1,\dots,20,$ respectively and Table \ref{tab: fishing selected costs} display the values of the cost function for each iteration. Figure \ref{fig:fishing} shows the relative rate of convergence of the cost function and the control.

\begin{figure}
\begin{center}
\includegraphics[width=.5\linewidth]{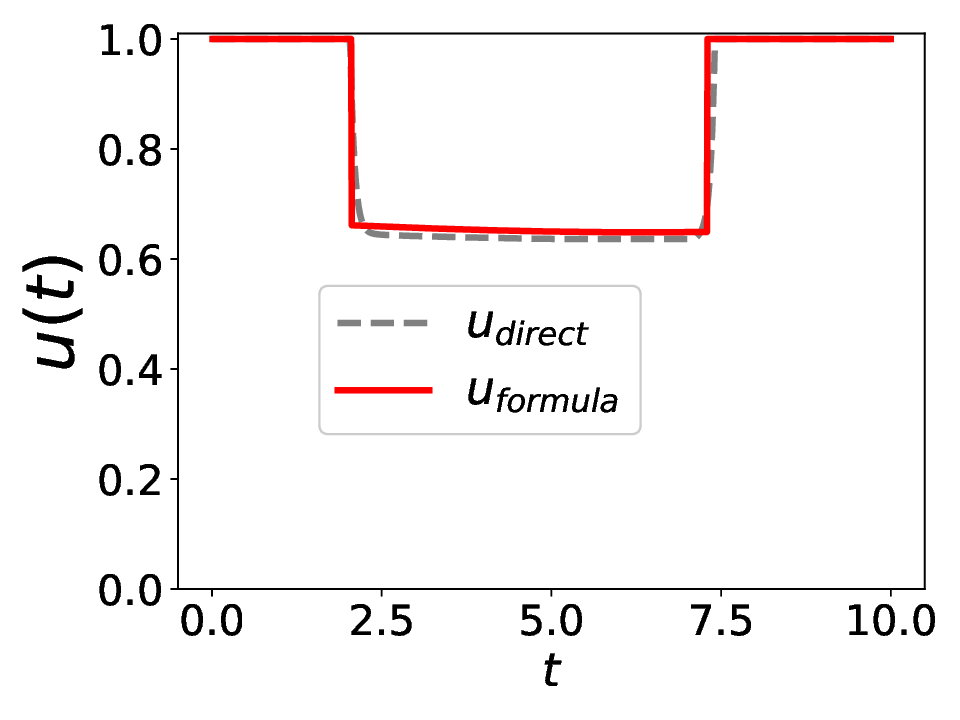}
\caption{In red, optimal control calculated with the formula derived in Theorem \ref{sing_control_scalar}. In dashed grey, the optimal control is calculated with the GEKKO library. }\label{fig:control_fishing}
\end{center}
\end{figure}
    
\begin{figure}[H]
    \centering
    \begin{subfigure}[b]{0.49\textwidth}
\includegraphics[width=\textwidth]{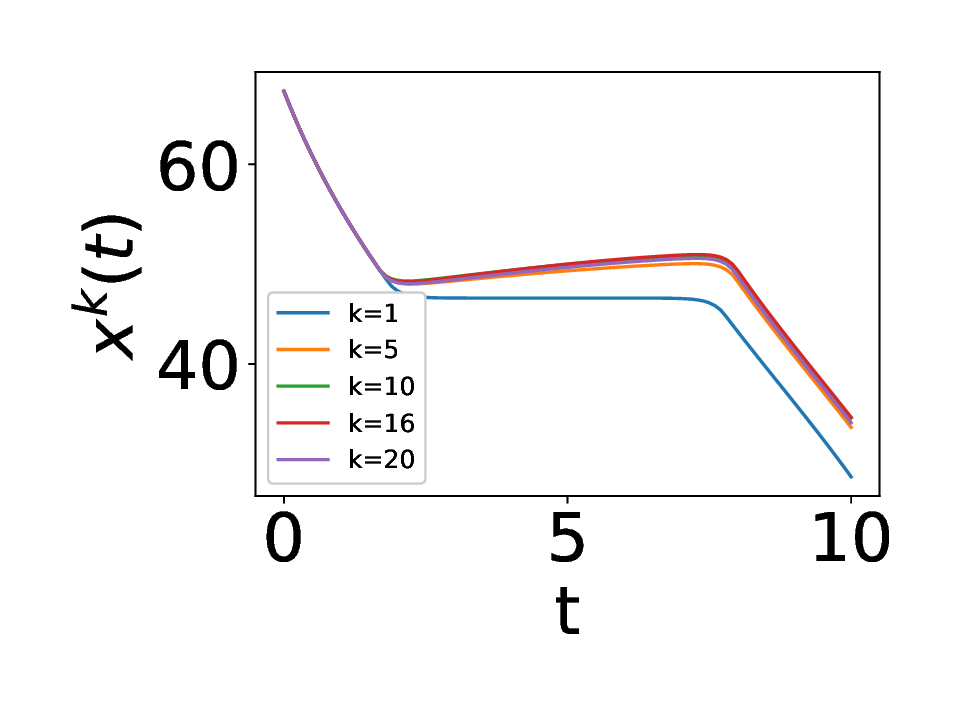}
        \caption{Trajectories for each time $t$ and selected iterations.}
        \label{fig:fishing trajectories}
    \end{subfigure}
    \hfill
    \begin{subfigure}[b]{0.49\textwidth}
    \includegraphics[width=\linewidth]{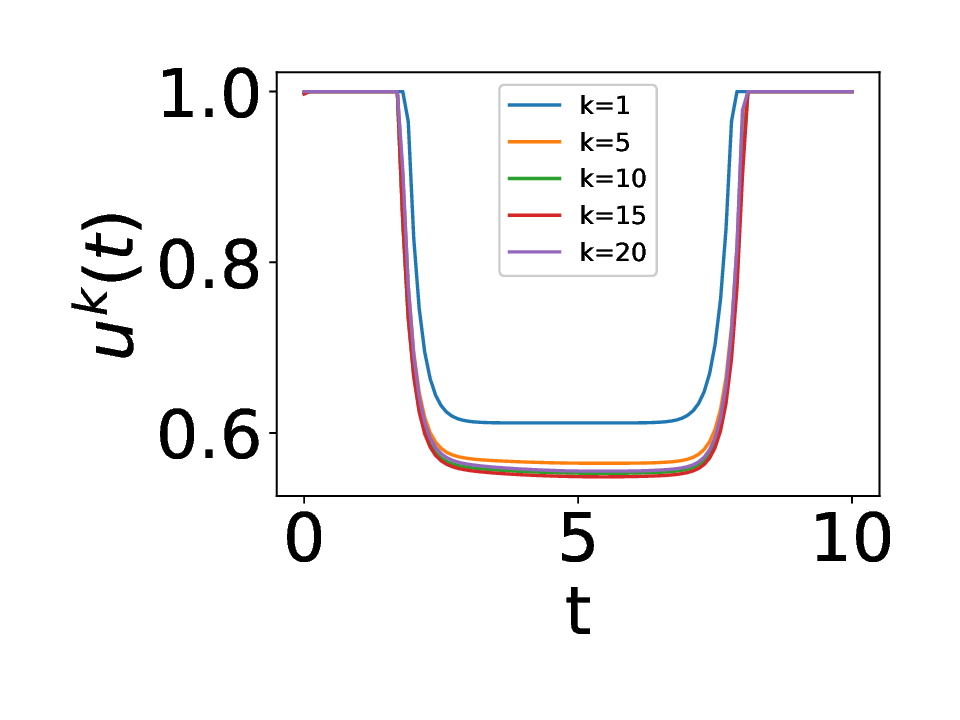}
        \caption{Control arc for each time $t$ and selected iterations.}
        \label{fig:fishing control arcs}
    \end{subfigure}
    \caption{Values for each time $t$ and selected iterations $k$.}
    \label{fig:fishing arc traj conv}
\end{figure}

\begin{table}[h!]
\centering
\begin{tabular}{|c|c|}
\hline
\textbf{Iteration} & \textbf{Final cost} \\ \hline
1 & 95.451 \\ \hline
5 & 93.053 \\ \hline
10 & 93.351 \\ \hline
15 & 93.523 \\ \hline
20 & 95.195 \\ \hline
\end{tabular}
\centering
\caption{Numerical values for the final cost in the given iteration.}
\label{tab: fishing selected costs}
\end{table}

\begin{figure}[H]
    \centering
    \begin{subfigure}[b]{0.45\textwidth}    \includegraphics[width=\textwidth]{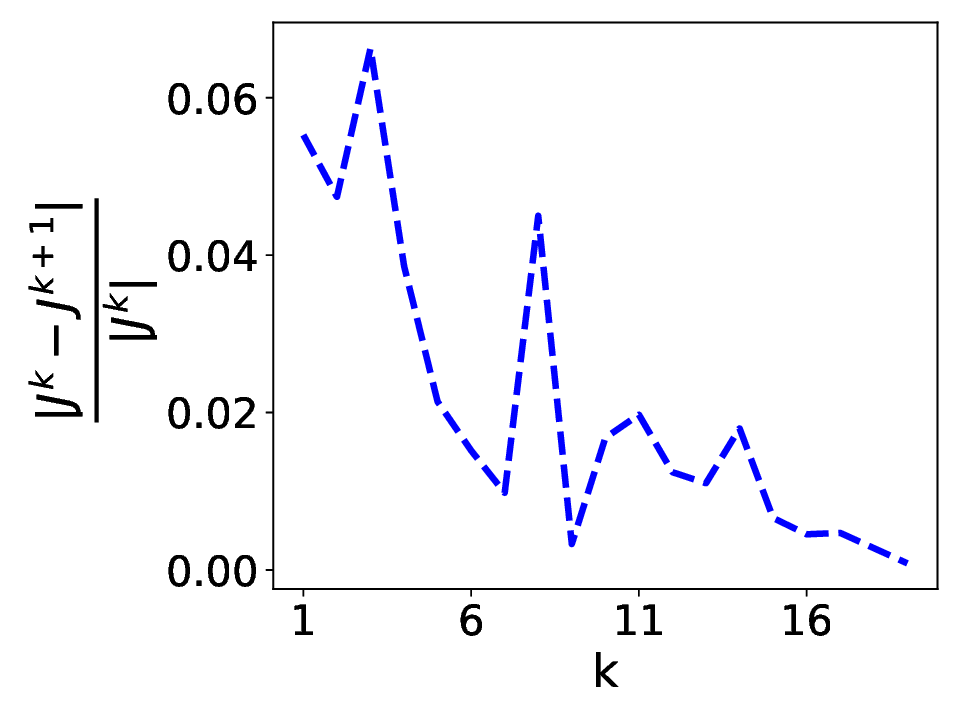}
        \caption{Cost relative distance between successive iterations}
        \label{fig:fishing cost}
    \end{subfigure}
    \hfill
    \begin{subfigure}[b]{0.45\textwidth}
        \includegraphics[width=\linewidth]{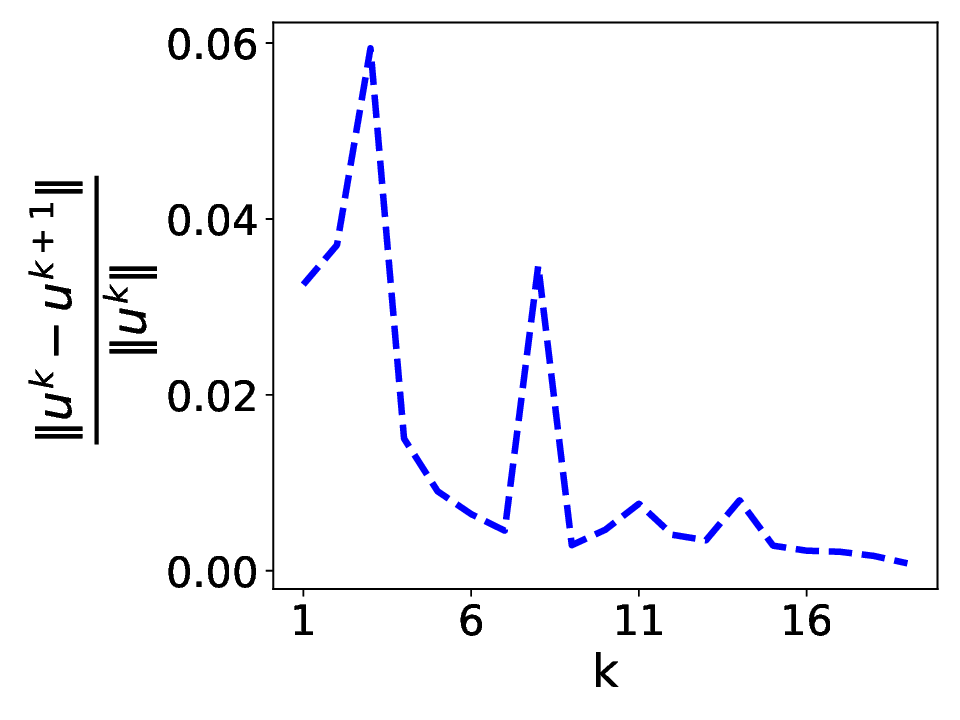}
        \caption{Control relative distance between successive iterations.}
        \label{fig:control fishing}
    \end{subfigure}
    \caption{Relative distances for cost and control.}
    \label{fig:fishing}
\end{figure}

\subsection{Example 2: A problem with vector control and commutative vector fields}

Consider the following variation of the problem \cite[page 248]{brysonHo}:
\if{
\begin{equation}
\begin{aligned}
&\displaystyle \text{minimize } J[u_1, u_2]= \int_\Omega \int_0^2  \frac{1}{2} x^2(t,\omega)d\mu(\omega)
\\
&\text{s.t}
\\
&\begin{cases}
    \begin{pmatrix}
        \dot x(t,\omega)\\
        \dot y(t,\omega)
    \end{pmatrix} =
    \begin{pmatrix}
        y(t,\omega)^2\\
        0
    \end{pmatrix}
    + u_1(t)    \begin{pmatrix}
        1\\
        0
    \end{pmatrix}
    + u_2(t)    \begin{pmatrix}
        0\\
        -1
    \end{pmatrix}  \text{ a.e } t\in [0,2]\\ \\
-1 \leq u_1(t),u_2(t) \leq 1 \text{ a.e on } [0,2]\\ \\
x(0,\omega) = y(0,\omega) \sim \text{Uniform}(0.95, 1.05)
\end{cases}
\end{aligned}
\end{equation}

As in the previous example, we introduce a new state $z$ to transform the problem into a Mayer one, so the new system can be written as}\fi
\begin{equation}
\label{varBH mayer}
\begin{aligned}
&\displaystyle \text{minimize } J[u_1, u_2]=  \int_\Omega z(2,\omega)  d\mu(\omega)\\
&\text{s.t}
\\
&\begin{cases}
\displaystyle 
    \begin{pmatrix}
        \dot x(t,\omega)\\
        \dot y(t,\omega)\\
        \dot z(t,\omega) 
    \end{pmatrix} =
    \begin{pmatrix}
        y^2(t,\omega)\\
        0\\
        \frac{1}{2} x^2(t,\omega)
    \end{pmatrix}
    + u_1(t)    \begin{pmatrix}
        1\\
        0\\
        0
    \end{pmatrix}
    + u_2(t)    \begin{pmatrix}
        0\\
        -1\\
        0
    \end{pmatrix},\quad \text{a.e. } t\in [0,2],
    \end{cases}
    \end{aligned}
\end{equation}
with $|u_1(t)|, |u_2(t)| \leq 1 \text{ a.e. } t \in [0,2]$ and $x(0,\omega) = y(0,\omega) \sim U(0.95, 1.05)$ the uniform distribution in $[0.95,1.05],$ which again correspond to $L^2$-functions, since their probability density functions are square-integrable.
The involved vector fields are given by
$$
f_1 := \begin{pmatrix}1\\0\\0\end{pmatrix} \ \text{and} \ f_2 :=\begin{pmatrix}0\\-1\\1\end{pmatrix}
$$
which are commutative (this is, $[f_1,f_2]$ identically vanishes) and, noting that, $[f_2,[f_0,f_1]]=[f_1,[f_0,f_2]]=0$, the singular controls $u^*_1(t)$ and $u^*_2(t)$ given by Theorem \ref{Thmfeedbackinfinite} read:

$$
u^*_1(t) =- \frac{1}{2}\frac{\displaystyle \int_\Omega p_z(t,\omega) y^2(t,\omega) d\mu(\omega)}{\displaystyle \int_\Omega p_z(t,\omega) d\mu(\omega)}
,\quad
u^*_2(t) =- \frac{\displaystyle \int_\Omega p_x(t,\omega) y^2(t,\omega) d\mu(\omega)}{\displaystyle  \int_\Omega p_x(t,\omega) d\mu(\omega)}
$$
where $p=(p_x,p_y,p_z)^\top$
is the adjoint state of the system. 

Once again we approximate the system under investigation \eqref{varBH mayer} using the discretization scheme  $(P_k)$ for sufficiently large values of $k$. For $k = 19$, we obtained 
\begin{gather*}
\frac{|J^{18} - J^{19}|}{|J^{18}|}=5.9\times10^{-4},\quad
\frac{\|u^{18}_1 - u^{19}_1\|}{\|u^{18}_1\|}=1.8\times10^{-2},\quad 
\frac{\| u^{18}_2 - u^{19}_2 \|}{\|u^{18}_2\|}=3.5\times10^{-2}.
\end{gather*}
Figures \ref{fig:bh control1} and \ref{fig:bh control2} exhibit the controls $u_1$ and $u_2$ as follows: the gray dotted function is the control approximated using the sequence of problems $(P_k)$ and the red one is the function computed via \eqref{mult dim vector}. The trajectories $x^k$ and $y^k$ in Figures \ref{fig:bh x arcs} and \ref{fig:bh y arcs}, respectively, are presented for each $t\in[0,2]$ and for selected iterations. In Table \ref{tab: bh selected costs}, values for the cost are presented for selected iterations. Figure \ref{fig:bh} exhibit the relative convergence rate for the cost functions and the controls.

\begin{figure}[H]
    \centering
    \begin{subfigure}{0.48\textwidth} 
        \centering
        \includegraphics[width=\linewidth]{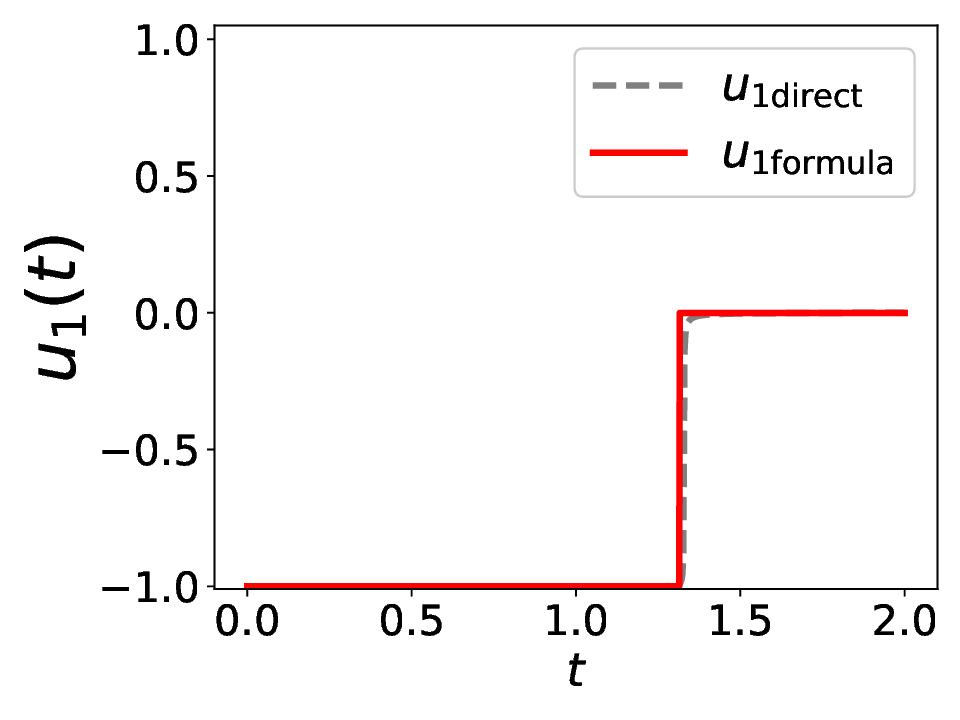}
        \caption{}
        \label{fig:bh control1}
    \end{subfigure}
    \hfill
    \begin{subfigure}{0.48\textwidth} 
        \centering
        \includegraphics[width=\linewidth]{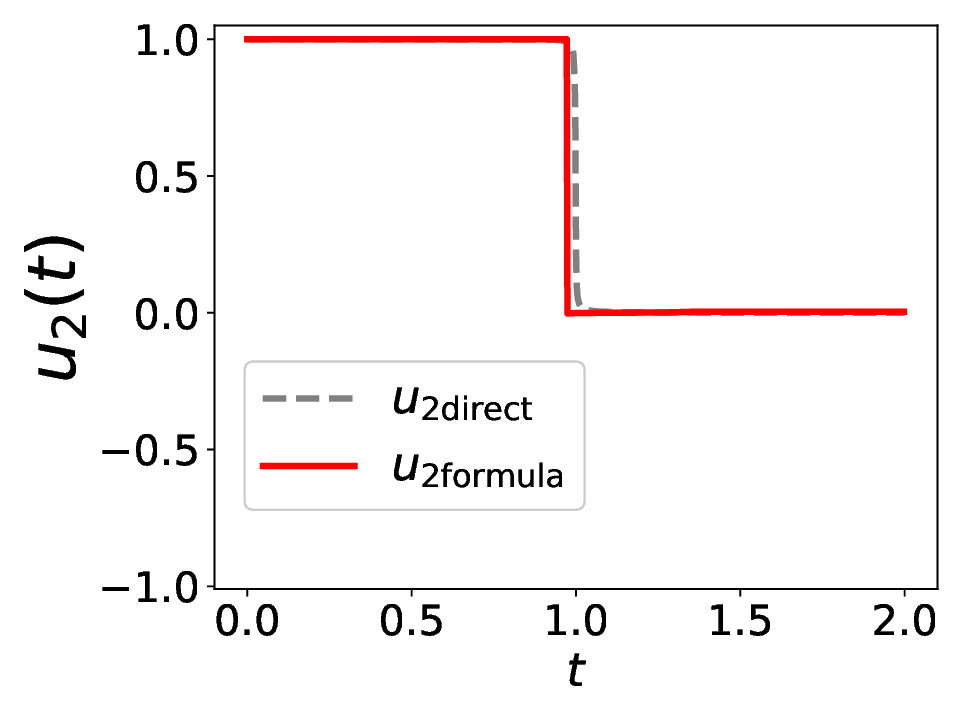}
        \caption{}
        \label{fig:bh control2}
    \end{subfigure}
    \caption{\small The dashed gray lines show the controls calculated with the GEKKO library, while the red lines show the controls calculated using \eqref{mult dim vector}.}
    \label{fig:bh controls comp}
\end{figure}

\begin{figure}[htbp]
    \centering
    \begin{subfigure}[b]{0.49\textwidth}
        \includegraphics[width=\linewidth]{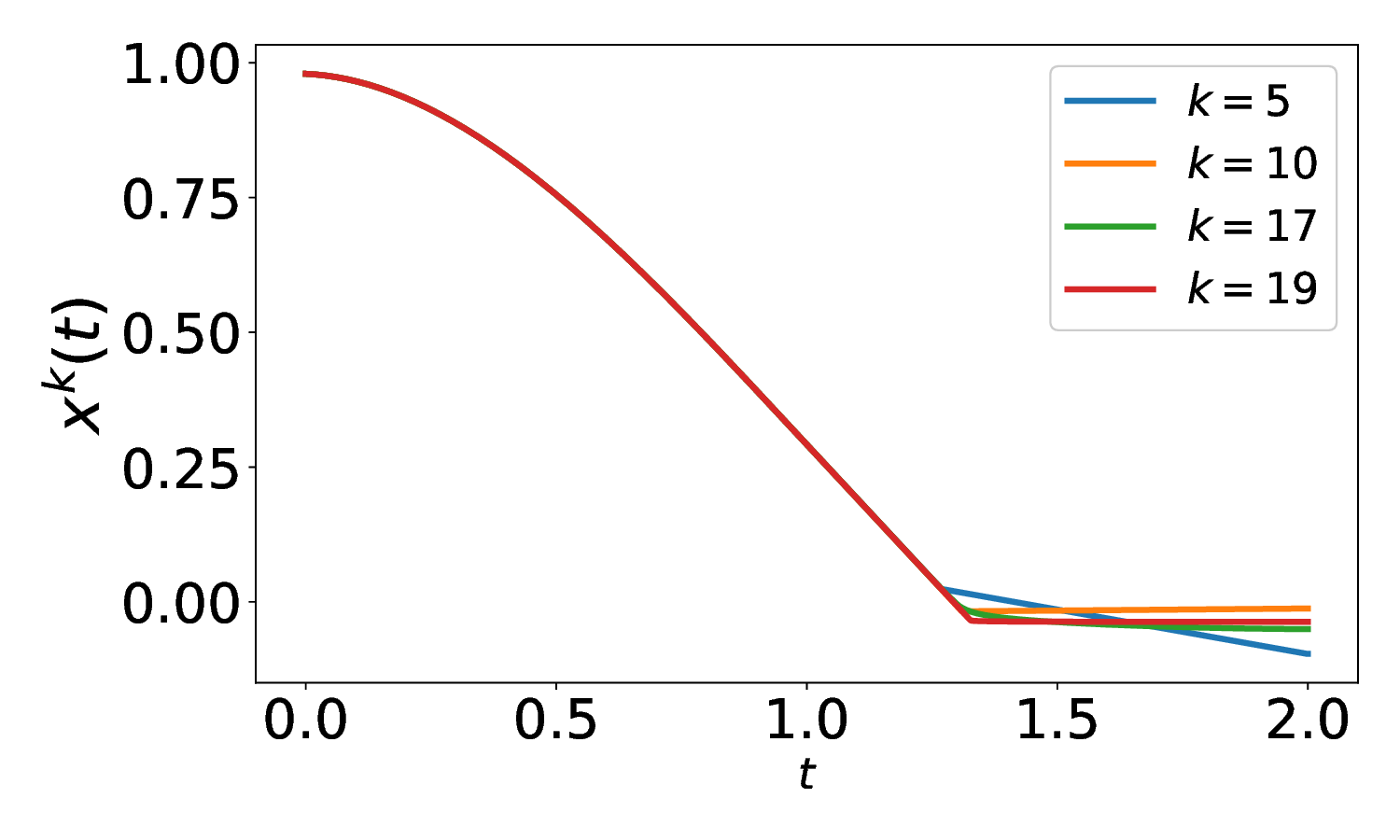}
        \caption{The arc $x^k(t)$ for each $t$ and selected iterations.}
        \label{fig:bh x arcs}
    \end{subfigure}
    \hfill
    \begin{subfigure}[b]{0.49\textwidth}
    \includegraphics[width=\linewidth]{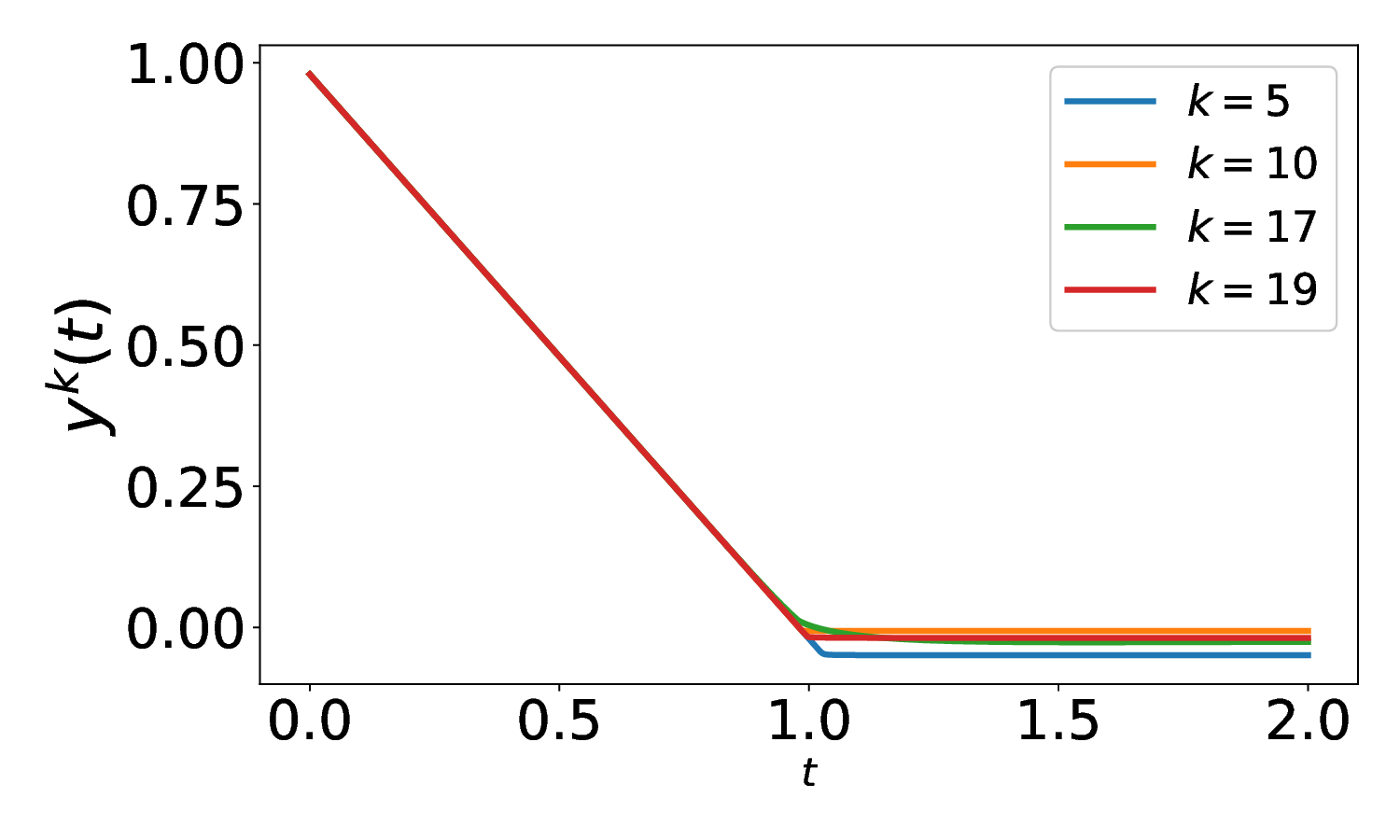}
         \caption{The arc $y^k(t)$ for each $t$ and selected iterations.}
        \label{fig:bh y arcs}
    \end{subfigure}
        \caption{Trajectories arcs for each state in each time $t$ and selected iterations k.}
    \label{fig:bh arcs}
\end{figure}

\begin{figure}[htbp]
    \centering
    \begin{subfigure}[b]{0.45\textwidth}
        \includegraphics[width=\linewidth]{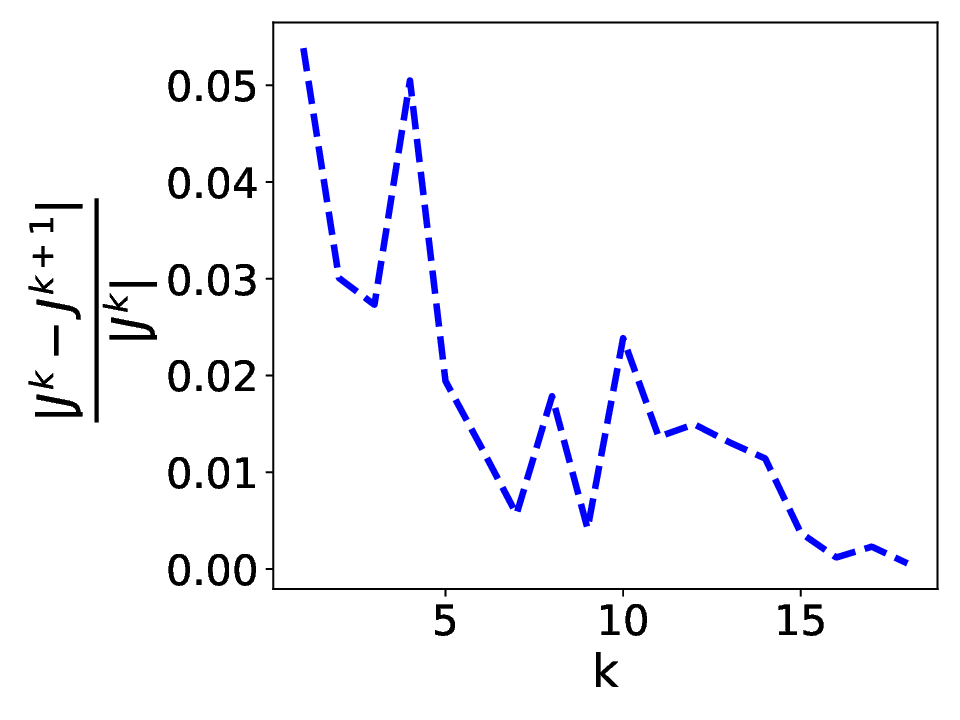}
        \caption{Cost relative distance between successive iteration.}
        \label{fig:bh cost}
    \end{subfigure}
    \hfill
    \begin{subfigure}[b]{0.45\textwidth}
        \includegraphics[width=\linewidth]{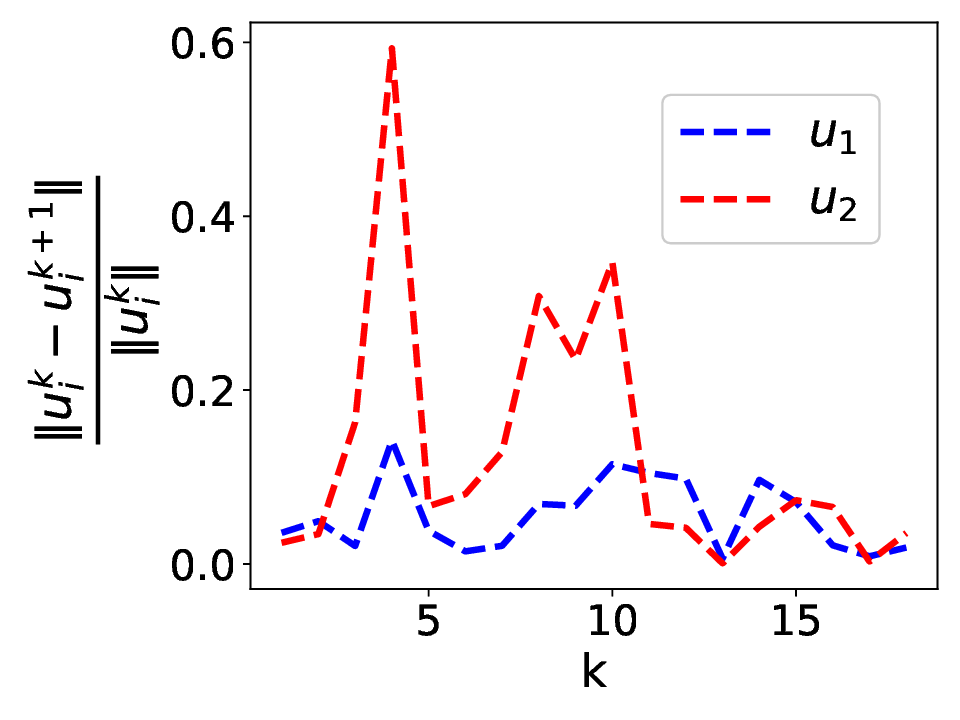}
        \caption{Control relative distance between successive iteration.}
        \label{fig:bh controls}
    \end{subfigure}
    \caption{Relative distances for cost and controls.}
    \label{fig:bh}
\end{figure}

\begin{table}[ht]
\centering
\begin{tabular}{|c|c|}
\hline
\textbf{Iteration} & \textbf{Final cost} \\ \hline
1 & 0.274 \\ \hline
5 & 0.286\\ \hline
10 & 0.290 \\ \hline
19 & 0.299 \\ \hline
\end{tabular}
\caption{Numerical values for the final cost in the given iteration.}
\label{tab: bh selected costs}
\end{table}

\section{Conclusions}

In this paper, we characterized the optimal control for the Riemann-Stieltjes optimal control problem by applying Pontryagin’s Maximum Principle to cases with both known and unknown initial conditions. For the former, we leveraged existing results from \cite{Bettiol2019}, while for the latter, we adopted the approach from \cite{Frankowska1992}, formulating the problem in an infinite-dimensional space where the initial condition is treated as an element of a Hilbert space. This framework enabled a rigorous analysis of control structures under uncertainty.  

Our numerical experiments confirmed the effectiveness of the proposed methodology. The computed control trajectories closely matched those obtained using Python’s GEKKO optimization package, providing validation of the theoretical findings. Additionally, the observed convergence of cost-relative and control-relative distances toward zero demonstrated the stability and reliability of the approach.  

These results highlight the robustness of our framework and its potential applicability to a broad range of uncertain optimal control problems. Future work could explore extensions to more general system dynamics, as well as the development of alternative numerical solution techniques to further enhance computational efficiency and applicability.

\begin{appendices}
\section{ Some technical lemmas }\label{AppA}

\subsection{Proof of the Lemma \ref{cont.traj}}

 {\it (i)} For all $u \in \mathcal{U}_{ad}[t_0,T]$ and $\omega\in\Omega,$  the corresponding trajectory satisfies the equation \eqref{trajectory}, then for all $t\in [t_0,T]$
$$
\big|x_{u}(t,\omega)\big| \leq \big| x_0\big| + \int_{t_0}^t \big|f_0( x_{u}(\sigma,\omega),\omega) + \sum_{i=1}^m f_i(x_{u}(\sigma,\omega),\omega)u_i(\sigma)\big| d\sigma,
$$
which, in view of \textbf{(H0)}, results in
\begin{equation}
\begin{aligned}
        \big|x_{u}(t,\omega)\big| &\leq \big| x_0\big| + \left(c_0+\|u\|_{L^{\infty}}\sum_{i=1}^mc_i\right)\int_{t_0}^t \big(1+ |x_{u}(\sigma,\omega)|\big) d\sigma,
        \\
        &= \big| x_0\big|+ \left(c_0+\|u\|_{L^{\infty}}\sum_{i=1}^mc_i\right)(t-t_0)
        \\
& \hspace{3cm}+\left(c_0+\|u\|_{L^{\infty}}\sum_{i=1}^mc_i\right)\int_{t_0}^t |x_{u}(\sigma,\omega)| d\sigma, 
\end{aligned}
\end{equation}
so by Grönwall's inequality, we get 
$$
\big|x_{u}(t,\omega)\big|\leq \left[\big| x_0\big|+ \left(c_0+\|u\|_{L^{\infty}}\sum_{i=1}^mc_i\right)(t-t_0)\right]e^{(t-t_0)\left(c_0+\|u\|_{L^{\infty}}\sum_{i=1}^mc_i\right)} 
$$
and {\it (i)} follows.

 {\it(ii)} Fix a control $u \in \mathcal{U}_{ad}[t_0,T].$      From \eqref{trajectory} we have that for all $\omega_{1},\omega_{2}\in \Omega$ and all $t\in [t_0,T],$
     \begin{equation*}
         \begin{aligned}
             &|x_u(t,\omega_{1})-x_u(t,\omega_{2})|\leq 
             \\ 
&\,\,\,\,\,\int_{t_0}^{t}\bigg(|f_0(x_u(\sigma,\omega_{1}),\omega_{1})-f_0(x_u(\sigma,\omega_{2}),\omega_{2})| \\
             &\hspace{3.2cm} +\sum_{i=1}^m|f_i(x_u(\sigma,\omega_{1}),\omega_{1})-f_i(x_u(\sigma,\omega_{2}),\omega_{2})|\,|u_i(\sigma)|\bigg)\,d\sigma \\
              &=\int_{t_0}^{t}\bigg(|f_0(x_u(\sigma,\omega_{1}),\omega_{1})-f_0(x_u(\sigma,\omega_{1}),\omega_{2})| \\
              &\hspace{4.5cm}+|f_0(x_u(\sigma,\omega_{1}),\omega_{2})-f_0(x_u(\sigma,\omega_{2}),\omega_{2})|\\
&+\sum_{i=1}^m\left[|f_i(x_u(\sigma,\omega_{1}),\omega_{1})-f_i(x_u(\sigma,\omega_{1}),\omega_{2})|\,|u_i(\sigma)|\right.
             \\
&\hspace{3.2cm}\left.+|f_i(x_u(\sigma,\omega_{1}),\omega_{2})-f_i(x_u(\sigma,\omega_{2}),\omega_{2})|\,|u_i(\sigma)|\right]\bigg)\,d\sigma,
         \end{aligned}
     \end{equation*} 
  conditions \textbf{  (H0)} and \eqref{f.regular} lead us to 
\begin{equation*}
         \begin{aligned}
             &|x_u(t,\omega_{1})-x_u(t,\omega_{2})|\leq 
(1+\|u\|_{L^{\infty}})\theta_f(\rho_\Omega(\omega_1,\omega_2)) \\
&\hspace{3.5cm}+\left(k_0+\|u\|_{L^{\infty}}\sum_{i=1}^{m}k_i\right)\int_{t_0}^{t}|x_u(\sigma,\omega_{1})-x_u(\sigma,\omega_{2})|\,d\sigma.
         \end{aligned}
     \end{equation*} 
 From Grönwall's inequality, we have that 
 \begin{equation*}
 \begin{aligned}
     |x_u(t,\omega_{1})-x_u(t,&\omega_{2})|\leq \\
&(1+\|u\|_{L^{\infty}})e^{(t-t_0) \left(k_0+\|u\|_{L^{\infty}}\sum_{i=1}^{m}k_i\right)}\theta_f(\rho_\Omega(\omega_1,\omega_2)),
\end{aligned}
 \end{equation*}
 and the continuity of the map $\omega \mapsto x_u(\cdot,\omega)$ follows.
 
 {\it (iii)} Fix $\omega\in\Omega$ and consider two different controls $u,v\in{\mathcal{U}}_{\rm ad}[t_0,T]
 $
 with their respective trajectories $x_{u}
(\cdot,\omega)$ and $x_{v}(\cdot,\omega).$ From \eqref{trajectory},  it follows that for all $t\in[t_0,T]$ 
  \begin{equation*}
         \begin{aligned}
             &|x_{u}(t,\omega)-x_{v}(t,\omega)|\leq \int_{t_0}^{t}\bigg(|f_0(x_{u}(\sigma,\omega),\omega)-f_0(x_{v}(\sigma,\omega),\omega)| \\
             &\hspace{3.5cm} +\sum_{i=1}^m|f_i(x_{u}(\sigma,\omega),\omega)u_i(\sigma)-f_i(x_{v}(\sigma,\omega),\omega)v_i(\sigma)|\bigg)\,d\sigma.
                        \end{aligned}
             \end{equation*}
    From \textbf{  (H0)}, {\it (i)} and  the Cauchy- Schwarz inequality, we have that
    \begin{equation*}\label{xuxv}
        \begin{aligned}
   |x_{u}(t,\omega)-x_{v}(t,\omega)|&\leq \left( k_0 +\|u\|_{L^{\infty}}\sum_{i=1}^mk_i\right)\int_{t_0}^{t}|x_{u}(\sigma,\omega)-x_{v}(\sigma,\omega)|\,d\sigma
             \\
&\hspace{1.4cm}+C\int_{t_0}^{t}\left(1+|(x_{v}(\sigma,\omega)|\right)|u_i(\sigma)-v_i(\sigma)|\,d\sigma
\\ &\leq C(t-t_0)^{1/2}\left(1+C_x\right)\|u-v\|_{L^{2}([t_0,T]:\mathbb{R}^m)}  \\
 &\ \ +\left( k_0+\|u\|_{L^{\infty}}\sum_{i=1}^mk_i\right)\int_{t_0}^{t}|x_{u}(\sigma,\omega)-x_{v}(\sigma,\omega)|\,d\sigma,
        \end{aligned}
    \end{equation*}
    then Grönwall's inequality leads us to 
{\small \begin{equation*}
        \begin{aligned}
 |x_{u}(t,\omega)-x_{v}(t,\omega)| 
 \leq C(t-t_0)^{1/2}\left(1+C_x\right)e^{(t-t_0)\left( k_0+\|u\|_{L^{\infty}}\sum_{i=1}^mk_i\right)}\|u-v\|_{L^{2}}.
        \end{aligned}
    \end{equation*}}
On the other hand, since $||u||_{L^2(t_0,t;\mathbb{R}^m)}\leq (t-t_0)^{1/2}||u||_{L^{\infty}(t_0,t;\mathbb{R}^m)},$ we have that
{\small  \begin{equation*}
        \begin{aligned}
 |x_{u}(t,\omega)-x_{v}(t,\omega)| 
 \leq C(t-t_0)\left(1+C_x\right)e^{(t-t_0)\left( k_0+\|u\|_{L^{\infty}}\sum_{i=1}^mk_i\right)}\|u-v\|_{L^{\infty}},
        \end{aligned}
    \end{equation*}}
       and the proof ends.
    \begin{flushright}
$\bf \square$
\end{flushright}

\begin{lemma}\label{lemma_xp}
For a given optimal process $(\bar u, \{\bar x(\cdot,\omega) : \omega \in \Omega \}),$ the following holds true,
\begin{enumerate}[i)]
    \item[{\it i)}] there exists a positive constants $C_p$  such that, $|p(t,\omega)| \leq C_p$ for all $t\in[0,T]$ and  all $\omega\in \mathrm{supp}(\mu),$ 
    \item[{\it ii)}] if $\varphi \equiv x_0,$ the map $\omega\mapsto p(\cdot,\omega)$ is continuous. 
\end{enumerate}
\end{lemma}
\begin{proof}
Let $p$ be as in Theorem \ref{pmp} or Theorem \ref{pmpInf}. By defining $q:=p(T-t)$ for $t\in[0,T]$ we can see that $q$ is solution of 
$$
\begin{cases}
 \dot q(t,\omega) = \left(\frac{\partial f}{\partial x}(\bar x(T-t,\omega), \bar u(T-t), \omega)\right)^{\top}q(t,\omega),\\ \hspace{5cm} \text{a.e} \, \,  t\in [0,T] \, \,  \text{and  for all }\omega \in \mathrm{supp}(\mu), \\
 q(0,\omega)=-\nabla_x g(\bar{x}(T,\omega),\omega), \text{ for all }\omega \in \mathrm{supp}(\mu).
\end{cases}
$$
Consequently, we have that for all $t\in[0,T]$
$$
q(t,\omega) = -\nabla_x g(\bar{x}(T,\omega),\omega)+\int_{T-t}^{T}\left(\frac{\partial f}{dx}(\bar x(\sigma,\omega), \bar u(\sigma), \omega)\right)^{\top}q(\sigma,\omega)\,d\sigma,
$$
then, from \textbf{ (H1)}-(iv) one has that
$$
|q(t,\omega)| \leq |\nabla_x g(\bar{x}(T,\omega),\omega)|+\left(C_{f_0}^1+\|\bar{u}\|_{L^{\infty}}\sum_{i=1}^mC_{f_i}^1\right)\int_{T-t}^T|q(\sigma,\omega)|\,d\sigma,
$$
now applying Grönwall's inequality results in
\begin{equation*}
|q(t,\omega)| \leq |\nabla_x g(\bar{x}(T,\omega),\omega)|e^{t\left(C_{f_0}^1+\|\bar{u}\|_{L^{\infty}}\sum_{i=1}^mC_{f_i}^1\right)}.
\end{equation*}
From \textbf{ (H1)}-(ii) the item {\it i)} follows.

To show item  {\it ii)} we argue as above to obtain  that, for all $\omega_1,\omega_2\in \Omega$ the following inequality holds
\begin{equation*}
\begin{aligned}
&|q(t,\omega_1)-q(t,\omega_2)|\leq |\nabla_x g(\bar{x}(T,\omega_1),\omega_1)-\nabla_x g(\bar{x}(T,\omega_2),\omega_2)|
\\
&+\int_{T-t}^{T}\left(|\nabla_x f(\bar x(\sigma,\omega_1), \bar u(\sigma), \omega_1)-\nabla_x f(\bar x(\sigma,\omega_2), \bar u(\sigma), \omega_2)|\,|q(\sigma,\omega_1)|\right.
\\
&\quad+\left.|\nabla_x f(\bar x(\sigma,\omega_1), \bar u(\sigma), \omega_1)|\,|q(\sigma,\omega_1)-q(\sigma,\omega_2)|\right)\,d\sigma,
\end{aligned}
\end{equation*}
since $\varphi \equiv x_0$, Lemma \ref{cont.traj} applies, ensuring the continuity of the map $\omega \mapsto \nabla_x g(\bar{x}(T, \omega), \omega)$. Thus, assumptions \textbf{(H1)}-(iii),(iv), and  {\it i)} together with the Grönwall's inequality lead us to the result.
\end{proof}

\section{Technical lemmas for uncertain initial condition}\label{AppA2} 
\begin{lemma}\label{x.ineq}
    Let us denote $x_{s, \varphi}$ as the solution of \eqref{trajectory} with initial time $s$ and data $\varphi.$ Let assumptions \textbf{  (H0)} hold. Then, for any $s$ and $\tau$ with $t_0 \leq s \leq \tau \leq T,$  $\varphi, \bar{\varphi} \in L^2(\mu,\Omega;\mathbb{R}^n)$, and any control $u \in \mathcal{U}_{ad}[t_0,T]$, the following inequalities hold
    $$\begin{aligned}
 &\textit{1.}) \
\|x_{s, \varphi}(t)\|_{L^ 2} \leq C(T)e^{c(t-s)}\left(
1+\|\varphi\|_{L^2} \right), \quad t \in[t_0, T] . \\ 
 &\textit{2.}) \ 
\|x_{s, \varphi}(t)-x_{s, \bar{\varphi}}(t)\|_{L^2}  \leq e^{k(t-s)}\|\varphi-\bar{\varphi}\|_{L^2} , \quad t \in[t_0, T].
\\ 
 &\textit{3.}) \ \|x_{\tau, \varphi}(t)-x_{s, \varphi}(t)\|_{L^2}  \leq C(T)e^{k(t-\tau)} \big(
 1+\|\varphi\|_{L^2} \big)(\tau-s), \quad  t \in[\tau, T].
\\ 
&\textit{4.}) \
\|x_{s, \varphi}(t)- x_{s, \varphi}(\tau)\|_{L^2}  \leq C(T)e^{c(t-s)} \big(
1+\|\varphi\|_{L^2} \big)|t-\tau|, \quad t \in[s, T] .
 \end{aligned}
 $$
\end{lemma}
\begin{proof}
It follows directly from a standard application of Grönwall's Lemma (see for instance \cite[Ch.6, Lemma 2.1]{liYong}). 
\end{proof}
\begin{lemma}\label{filipschitz}
    Let $x\in C([t_0,T]:\Vspace)$ be a solution of \eqref{trajectory} for some control $u \in \mathcal{U}_{\rm ad}[t_0,T].$ Then, the maps $t\rightarrow f_i(x(t),\cdot)$ for all $i=0,...,m$ are Lipschitz continuous.   
\end{lemma}
\begin{proof}
  Fix $s,t \in [t_0,T],$ then, from \eqref{fLipschitz} and Lemma \ref{x.ineq} -{\it 4.)} we have that 
  $$
  ||f_i(x(t),\cdot)-f_i(x(s),\cdot)||_{L^2}\leq k_i||x(t)-x(s) ||_{L^2}\leq C(T)|t-s|
  $$
  for all $i=0,...,m.$
\end{proof}


\subsection{Proof of Lemma \ref{FJlemma}}
 For all $ \psi \in \Vspace,$ with $\|\psi\|_{L^2}=1,$ we claim that

\begin{equation*}
\frac{|\mathcal{J}(\varphi+\varepsilon\psi)-\mathcal{J}(\varphi)-\langle\nabla_xg(\varphi(\cdot),\cdot),\varepsilon\psi(\cdot)\rangle|}{\varepsilon}
\end{equation*}
tends to $0$ as $\varepsilon\rightarrow 0.$ Indeed, the whole term above can be written as
\small{\begin{equation*}
\begin{aligned}
&\frac{1}{\varepsilon}\bigg| \int_{\Omega}\,[g(\varphi(\omega)+\varepsilon\psi(\omega),\omega)-g(\varphi(\omega),\omega) ]\,d\mu(\omega)-\langle\nabla_xg(\varphi(\cdot),\cdot),\varepsilon\psi(\cdot)\rangle\bigg| \\
&=\frac{1}{\varepsilon}\bigg| \int_{\Omega} \int_{0}^{1}\frac{d}{d\sigma}g(\varphi(\omega)+\sigma\varepsilon\psi(\omega),\omega) \,d\sigma \,d\mu(\omega)-\langle\nabla_xg(\varphi(\cdot),\cdot),\varepsilon\psi(\cdot)\rangle\bigg| \\
&=\frac{1}{\varepsilon}\bigg| \int_{\Omega} \int_{0}^{1} \nabla_x g(\varphi(\omega)+\sigma\varepsilon\psi(\omega),\omega)\cdot \varepsilon\psi(\omega)  \,d\sigma -\nabla_xg(\varphi(\omega),\omega)\cdot\varepsilon\psi(\omega)\,d\mu(\omega)\bigg|\\
&= \bigg|\int_{\Omega} \int_{0}^{1} \nabla_x g(\varphi(\omega)+\sigma\varepsilon\psi(\omega),\omega)\cdot \psi(\omega)  \,d\sigma-\nabla_xg(\varphi(\omega),\omega)\cdot\psi(\omega)\,d\mu(\omega)\bigg|,
\end{aligned}
\end{equation*}}
\normalsize since $\nabla_x g(\cdot,\cdot) $ is continuous, the claim follows, and \eqref{FréchetJ} holds true.
\begin{flushright}
$\bf \square$
\end{flushright}





\end{appendices}

\section*{Acknowledgements} The authors acknowledge the support of the Brazilian agencies FAPERJ (Rio de Janeiro)  for processes E-28/201.346/2021, 210.037/2024, SEI 260003/000175/2024
and  CAPES 88887.488134/2020-00 and CNPq 312407/2023-8.
\\

\noindent
{\bf Data Availability.} It does not apply to this article as no new data were created or analyzed in
this study.
\\

\noindent
{\bf Conflict of interest.} This work does not have any conflicts of interest.



\end{document}